\documentclass[11pt]{amsart}
\usepackage{amsmath,amssymb,amsthm}
\usepackage{graphicx,float,url,color}
\usepackage{mathrsfs}
\usepackage{enumitem}
\usepackage{soul}
\usepackage[normalem]{ulem}
\usepackage{cancel}
\usepackage{subcaption}
\usepackage[colorlinks=true]{hyperref}
\hypersetup{
	colorlinks=true,    
	linkcolor=blue,    
	citecolor=blue,    
	filecolor=magenta,   
	urlcolor=blue     
}

\usepackage[left=1in,right=1in,top=1in,bottom=1in]{geometry}
\usepackage{etoolbox}

\patchcmd{\section}{\scshape}{\bfseries}{}{}

\newcommand{\noperiod}[1]{}

\makeatletter
\renewcommand{\section}{\@startsection{section}{1}%
 \z@{.7\linespacing\@plus\linespacing}{.5\linespacing}%
 {\normalfont\large\bfseries\centering}} 
\renewcommand{\@secnumfont}{\bfseries} 
\makeatother

\def\R{\mathrm{I\kern-0.21emR}}
\def\N{\mathrm{I\kern-0.21emN}}

\newcommand{\E}{\mathbb{E}}

\newcommand{\Qbb}{\mathbf{Q}}
\newcommand{\W}{\mathbf{W}}
\newcommand{\Ent}{\operatorname{Ent}}
\newcommand{\dd}{\,\mathrm{d}}
\newcommand{\Prob}{\mathcal{P}}
\newcommand{\Law}{\operatorname{Law}}
\newcommand{\Cov}{\operatorname{Cov}_{\eta_\lambda}}
\newcommand{\Covtx}{\Cov(t,x)}  

\renewcommand{\geq}{\geqslant}
\renewcommand{\leq}{\leqslant}
\renewcommand{\Vert}{|}

\numberwithin{equation}{section}
\numberwithin{figure}{section}

\newtheorem{theorem}{Theorem}[section] 
\newtheorem{proposition}[theorem]{Proposition}
\newtheorem{corollary}[theorem]{Corollary}
\newtheorem{lemma}[theorem]{Lemma}
\theoremstyle{definition}\newtheorem{remark}[theorem]{Remark}
\newtheorem{assumption}[theorem]{Assumption}

\title[Optimal drift optimizer for non-convex optimization]{Optimal drift optimizer for non-convex optimization}

\author{Qin Li}
\address{Department of Mathematics, University of Wisconsin-Madison, Madison WI 53706, USA}
\email{qinli@math.wisc.edu}

\author{Sixu Li}
\address{Department of Mathematics, University of Wisconsin-Madison, Madison WI 53706, USA}
\email{sli739@wisc.edu}

\author{Eitan Tadmor}
\address{Department of Mathematics and Institute for Physical Science \& Technology\newline \hspace*{0.3cm} University of Maryland, College Park, USA}
\email{tadmor@umd.edu}

\author{Emmanuel Tr\'{e}lat}
\address{Sorbonne Universit\'e, Universit\'e Paris Cit\'e, CNRS, Inria\newline \hspace*{0.3cm} Laboratoire Jacques-Louis Lions, LJLL, F-75005 Paris, France}
\email{emmanuel.trelat@sorbonne-universite.fr}

\date{\today}



\thanks{\textbf{Acknowledgment.} 
Research of Q. Li and S. Li was supported by NSF-DMS 2308440 and Vilas Associate Award. Research of E. Tadmor was supported by ONR grant N00014-2412659, by NSF grant DMS-2508407 and by the Fondation Sciences Math\'ematiques de Paris (FSMP) while being hosted by
LJLL at Sorbonne University. E. Tr\'elat acknowledges the support by FA8655-25-1-7469 of the EOARD-AFOSR}

\begin{document}
\begin{abstract}
We study a finite-horizon stochastic control criterion for non-convex optimization in which Brownian exploration is balanced against a quadratic control cost. Rather than emphasizing the classical Hopf--Cole representation, we isolate the exact drift selected by the criterion and reorganize it in a form adapted to optimization. The key object is the conditional terminal law of the optimal process. We show that this law is a Gibbs measure for a proximally penalized energy, yielding three exact representations of the drift: potential, averaged-gradient, and barycentric. We then analyze two asymptotic regimes relevant for optimization. As terminal time is approached, the drift recovers a scaled gradient-descent field. In the low-temperature regime, assuming a unique global minimizer, the conditional terminal law concentrates on it even in the presence of nonglobal local minima, and the drift converges to an affine attraction field toward it. In the nondegenerate case we also derive Laplace asymptotics for the drift, the value function, and the covariance of the conditional terminal law. Finally, we record a simple gradient-free discretization suggested by the barycentric formula.
\end{abstract}
\maketitle

\setcounter{tocdepth}{1}
\tableofcontents

\section{Introduction}\label{intro_sec}
We are interested in the global minimization problem
$$
x^* \in \operatorname*{argmin}_{x \in \R^d} f(x),
\qquad
f_* = \min_{x \in \R^d} f(x),
$$
for a possibly non-convex objective function $f : \R^d \to \R$.

The starting point of the paper is the finite-horizon stochastic control criterion
\begin{equation}\label{criterion_intro}
J^\lambda(u) = \E \bigg[ f(X_T) + \frac{\lambda}{2T}\int_0^T \Vert u_t \Vert^2 \dd t \bigg],
\quad
X_t = x_0 + \int_0^t u_s \dd s + \sqrt{2\beta}B_t, \quad u_s = \mathbf u(s,X_s).
\end{equation}
We use \emph{control}, \emph{drift}, and --- when speaking from the optimization viewpoint --- \emph{optimizer}, for the process $u$. In the deterministic warm-up below, $u$ is an open-loop control depending only on time. In the stochastic problem, the relevant objects will be time-dependent Markov feedbacks of the form $u_t = \mathbf u(t,X_t)$, which is standard terminology in finite-horizon control theory (see, e.g., \cite{FlemingSoner06, TrelatBook, YongZhou99}). Throughout the paper, $u_s$ denotes the realized admissible control process,  $\mathbf u(t,x)$ denotes a Markov feedback vector field, while $\mathbf u_\lambda^*$ will denote the optimal feedback selected by the finite-horizon criterion associated with $J^\lambda$ (see Theorem \ref{main_value} below).
We keep the expression ``optimal drift optimizer'' as a convenient shorthand for the feedback selected by \eqref{criterion_intro}, but always relative to this precise criterion. The paper does not claim a universal optimal algorithm for non-convex optimization. 

From the viewpoint of stochastic control, the analytic core of \eqref{criterion_intro} is classical. The logarithmic transform reducing the Hamilton-Jacobi-Bellman equation to a linear parabolic equation belongs to the linearly solvable control literature \cite{Kappen05,Todorov06,Todorov09}. Closely related exponential reweightings of Brownian motion appear in reciprocal diffusion and Schr\"odinger-F\"ollmer constructions \cite{DaiJiaoKangLuYang23, DaiPra91, Follmer85, HuangJiaoKangLiaoLiuLiu21, TzenRaginsky19, ZhangChen22}. The same forward-backward structure also appears in mean-field games \cite{BensoussanFrehseYam13, LasryLions07}. On the PDE and optimization side, related kernels have recently been interpreted through Wasserstein proximal operators and Hamilton-Jacobi regularizations \cite{HuangKaliseKouhkouh25,HeatonFungOsher24, LiLiuOsher23, OsherHeatonFung23}. The present viewpoint should also be compared with Gaussian homotopies, consensus-based optimization, and gradient-free integration methods \cite{AndrieuChopinFincatoGerber24, CarrilloChoiTotzeckTse18, FornasierKlockRiedl24, IwakiriWangItoTakeda22}.

The closed form itself is therefore not the main contribution of the paper. The point of the present note is to isolate the \emph{optimization content} of that classical formula in a short and coherent framework on $\R^d$. More precisely, the paper emphasizes the following facts.
\begin{enumerate}[label=(\roman*)]
\item The criterion \eqref{criterion_intro} fixes its own scales. It selects both the Gaussian variance $2\beta(T-t)$ and the effective temperature $\frac{2\lambda\beta}{T}$, instead of taking them as external schedules. This distinguishes the present construction from Gaussian homotopy and consensus-type methods, where the smoothing scale or the consensus temperature is prescribed by hand.
\item The stochastic problem is best read as a feedback optimal control problem, and the law-level Fokker-Planck formulation explains naturally why a feedback field appears. The deterministic open-loop toy problem already contains the proximal mechanism behind the whole construction.
\item The central object is the conditional terminal law of the optimal process. It is a Gibbs law for a penalized energy, and it simultaneously yields the three useful forms of the drift: potential, averaged-gradient, and barycentric.
\item The low-temperature regime $\lambda \to 0$ is the global-optimization regime. Under the sole assumption that the global minimizer is unique, with no exclusion of nonglobal local minima, the terminal law concentrates on that minimizer and the drift becomes an affine attraction toward it. The regime $t \to T$ is the local consistency regime: it shows that the same drift reconnects with gradient descent near the end of the horizon. The two limits are generically non-commutative, reflecting the tension between global exploration and local exploitation.
\end{enumerate}

The main text is written on $\R^d$ with an integrable Gibbs density. We do not mix a free diffusion on $\R^d$ with a hard state constraint encoded by setting $f = +\infty$ outside a bounded set. The bounded-domain framework, namely reflected diffusions and no-flux boundary conditions, is mentioned as a perspective in Section \ref{further_comments_sec}.

Before turning to the stochastic feedback problem, it is useful to begin with the deterministic open-loop toy model behind the whole construction.

\begin{proposition}[Deterministic warm-up]\label{warmup}
Fix $x_0 \in \R^d$, $T > 0$, and $\lambda > 0$. Consider
$$
\inf \bigg\{ f(x(T)) + \frac{\lambda}{2T}\int_0^T \Vert u(t) \Vert^2 \dd t \ \ \big\vert\ \ \dot x(t) = u(t),\ x(0) = x_0 \bigg\}.
$$
Then
\begin{equation}\label{optim_warmup}
\inf_{u \in L^2(0,T;\R^d)} \bigg\{ f(x(T)) + \frac{\lambda}{2T}\int_0^T \Vert u(t) \Vert^2 \dd t \bigg\}
= \inf_{y \in \R^d} \bigg\{ f(y) + \frac{\lambda}{2T^2}\Vert y - x_0 \Vert^2 \bigg\}.
\end{equation}
If the right-hand side admits a minimizer $y_\lambda(x_0)$, then an optimal control is the constant control
$$
u_\lambda(t) = \frac{y_\lambda(x_0) - x_0}{T},
\qquad
x_\lambda(t) = x_0 + \frac{t}{T}\big( y_\lambda(x_0) - x_0 \big).
$$
\end{proposition}

\begin{proof}
Let $y = x(T)$. Since $y - x_0 = \int_0^T u(t)\dd t$, the Cauchy-Schwarz inequality gives
$$
\int_0^T \Vert u(t) \Vert^2 \dd t \geq \frac{1}{T}\bigg\Vert \int_0^T u(t)\dd t \bigg\Vert^2 = \frac{1}{T}\Vert y - x_0 \Vert^2,
$$
with equality if and only if $u$ is constant. Therefore, among all trajectories ending at $y$, the minimal kinetic cost equals $\frac{\lambda}{2T^2}\Vert y - x_0 \Vert^2$. Minimizing over $y$ gives the result.
\end{proof}

Proposition \ref{warmup} already shows what the quadratic running cost does: it turns the endpoint optimization into a proximal penalization. Equivalently, the right-hand side of \eqref{optim_warmup} is the Moreau envelope of $f$ with parameter $\mu = T^2/\lambda$, evaluated at $x_0$. The stochastic problem below is its noisy feedback analogue, with a terminal law instead of a single terminal point.

The same warm-up already contains the basic low-$\lambda$ intuition.

\begin{corollary}[Low-$\lambda$ limit in the warm-up]\label{warmup_small_lambda}
Assume that $f$ is continuous, bounded from below, coercive, and has a unique minimizer $x^*$. Let $y_\lambda(x_0)$ be any minimizer of the problem \eqref{optim_warmup}. Then $y_\lambda(x_0) \to x^*$ as $\lambda \downarrow 0$. Consequently, $u_\lambda(t) \to \frac{x^* - x_0}{T}$ for every $t \in [0,T]$.
\end{corollary}

\begin{proof}
For every $\lambda > 0$,
$$
f\big( y_\lambda(x_0) \big) + \frac{\lambda}{2T^2}\Vert y_\lambda(x_0) - x_0 \Vert^2 \leq f(x^*) + \frac{\lambda}{2T^2}\Vert x^* - x_0 \Vert^2.
$$
Hence
$$
f\big( y_\lambda(x_0) \big) \leq f_* + \frac{\lambda}{2T^2}\Vert x^* - x_0 \Vert^2.
$$
By coercivity, the family $\{ y_\lambda(x_0) \}_{\lambda > 0}$ is bounded. Let $\lambda_n \downarrow 0$ and assume $y_{\lambda_n}(x_0) \to \bar y$ along a subsequence. Passing to the limit in the previous inequality gives $f(\bar y) \leq f_*$. Hence $\bar y = x^*$ by uniqueness of the minimizer. Every convergent subsequence has the same limit, so the whole family converges to $x^*$.
\end{proof}

The deterministic warm-up is open-loop. The finite-horizon optimizer selected by \eqref{criterion_intro} will instead be a time-dependent Markov feedback. 

\medskip

The rest of the paper is organized as follows. Section \ref{control_sec} solves the finite-horizon optimal control problem and identifies the exact terminal Gibbs law and the three equivalent forms of the drift. Section \ref{asymp_sec} contains the asymptotic analysis: the regime $t \to T$ explains the local gradient behavior, while the regime $\lambda \to 0$ gives the global-selection mechanism and the Laplace asymptotics. Section \ref{numerics_sec} records a gradient-free discretization. Section \ref{further_comments_sec} concludes with further comments and perspectives.

\section{Optimal control and exact drift}\label{control_sec}
This section solves the finite-horizon optimal control problem associated with \eqref{criterion_intro}, more generally with the translated criterion 
$J^\lambda_{t,x}$ (defined by \eqref{def_Jtx} further) starting from an arbitrary time-space point $(t,x)$. We record both the Hamilton-Jacobi-Bellman and the law-level Fokker-Planck formulations, and we identify the conditional terminal law that carries the whole drift. The closed form itself is classical; what matters here is the way it is reorganized for optimization.

\subsection{Fokker-Planck formulation and Gibbs objects}

We fix throughout two parameters $T > 0$ and $\beta > 0$.

\begin{assumption}\label{basic_ass}
The objective function $f : \R^d \to \R$ satisfies the following conditions:
\begin{enumerate}[label=\roman*)]
\item $f \in C^2(\R^d)$, $f$ is bounded from below, and $f(x) \to +\infty$ as $\Vert x \Vert \to +\infty$;
\item for every $c > 0$, the functions $e^{-cf}$, $\Vert \nabla f \Vert e^{-cf}$, and\footnote{Here $\| \cdot \|$ denotes the usual induced matrix norm.} $\| D^2 f \| e^{-cf}$ belong to $L^1(\R^d)$.
\end{enumerate}

\end{assumption}

Assumption \ref{basic_ass} keeps the whole discussion on $\R^d$ and avoids artificial hard walls. It is satisfied, for instance, by coercive $C^2$ functions whose first two derivatives have at most polynomial growth. The low-temperature concentration arguments of Section \ref{asymp_sec} use less than this, but these assumptions are convenient for the exact formulas. The bounded-domain case with reflection is mentioned as a perspective in Section \ref{further_comments_sec}.

On a filtered probability space carrying a $d$-dimensional Brownian motion $B$, an admissible control on $[t,T]$ from $x$ is a progressively measurable process $(u_s)_{s \in [t,T]}$ with values in $\R^d$ such that $\displaystyle \E \int_t^T \Vert u_s \Vert^2 \dd s < +\infty$.
Given such a control, we set
$$
X_s = x + \int_t^s u_r \dd r + \sqrt{2\beta}(B_s - B_t),
\qquad
s \in [t,T],
$$
and
\begin{equation}\label{def_Jtx}
J^\lambda_{t,x}(u) = \E \bigg[ f(X_T) + \frac{\lambda}{2T}\int_t^T \Vert u_s \Vert^2 \dd s \bigg].
\end{equation}
A Markov control, also called a time-dependent feedback control in the finite-horizon setting, is a control of the special form $u_s = \mathbf u(s,X_s)$ for some Borel map $\mathbf u : [t,T] \times \R^d \to \R^d$ (see \cite{FlemingSoner06, YongZhou99}).

If the control is Markov and $\mathbf u$ is regular enough, then the law $\rho(s,\cdot)$ of $X_s$ solves the Fokker-Planck equation
\begin{equation}\label{FP}
\partial_s \rho + \nabla \cdot \big( \rho \, \mathbf u \big) = \beta \Delta \rho
\end{equation}
in the distributional sense; under smoothness assumptions the derivation is classical, see, for example, \cite[Chapter 6]{Pavliotis14} or \cite[Chapters 4 and 5]{Risken96}. Consequently, the criterion $J^\lambda_{t,x}$ may be rewritten at the level of densities as
\begin{equation}\label{FP_control}
\inf_{\mathbf u} \bigg\{ \int_{\R^d} f(y)\rho(T,y)\dd y + \frac{\lambda}{2T}\int_t^T \int_{\R^d} \Vert \mathbf u(s,y) \Vert^2 \rho(s,y)\dd y\dd s \bigg\},
\end{equation}
subject to the dynamical constraint \eqref{FP}. This formulation is the natural feedback analogue of Proposition \ref{warmup}.
Thus, starting from an initial density $\rho(t,\cdot) = \rho_0$, one seeks a feedback vector field $\mathbf u$ and an associated path of probability densities $\rho(s,\cdot)$ solving \eqref{FP}. The density is not an additional modeling assumption; it is the law of the controlled diffusion and therefore records the evolution of mass in configuration space. In the global-optimization regime described in Section \ref{asymp_sec}, the terminal mass is expected to concentrate near $x^*$. This is the sense in which the drift selected by \eqref{FP}--\eqref{FP_control} may be called an optimal drift optimizer.

\subsection{Exact value function and optimal feedback}

We begin with standard notation. For $r > 0$, let
$$
G_r(z) = \frac{1}{(4\pi \beta r)^{d/2}} \exp \bigg( - \frac{\Vert z \Vert^2}{4\beta r} \bigg)
$$
denote the heat kernel with diffusivity $\beta$.
Fix $\lambda > 0$ and let 
\begin{equation*}
\pi_\lambda(y) = \frac{1}{M_\lambda} e^{-\alpha_\lambda f(y)}, \qquad \alpha_\lambda = \frac{T}{2\lambda\beta},
\quad
M_\lambda = \int_{\R^d} e^{-\alpha_\lambda f(y)}\dd y.
\end{equation*}
$\pi_\lambda$ denotes the normalized Gibbs density associated with $f$ at inverse temperature $\alpha_\lambda$, or equivalently at effective temperature $\frac{2\lambda\beta}{T}$; see, for example, \cite[Chapter 2]{DemboZeitouni}. In the present setting these quantities are not inserted by hand: they are selected by the Hopf-Cole transform of the HJB equation, as in the linearly solvable control literature \cite{Kappen05, Todorov06, Todorov09}, and they also arise from the static free-energy problem
\begin{equation*}
- \frac{2\lambda\beta}{T}\log M_\lambda
= \inf \bigg\{ \int_{\R^d} f(y)\rho(y)\dd y + \frac{2\lambda\beta}{T}\int_{\R^d} \rho(y)\log \rho(y)\dd y \ \mid\ \rho \geq 0,\ \int_{\R^d} \rho(y)\dd y = 1 \bigg\},
\end{equation*}
whose unique minimizer is $\pi_\lambda$ (see \cite[Chapter 4]{DemboZeitouni}). In this sense $\pi_\lambda$ is the Gibbs law that balances objective value and entropy at temperature $\frac{2\lambda\beta}{T}$.

Finally, for $0 \leq t < T$ and $x \in \R^d$, we define

\begin{equation}\label{def_h_phi}
h_\lambda(t,x) = (G_{T-t} * \pi_\lambda)(x),
\qquad
\Phi_\lambda(t,x) = - 2\beta \log h_\lambda(t,x),
\end{equation}
and
\begin{equation}\label{def_V_lambda}
V_\lambda(t,x) = \frac{\lambda}{T}\Phi_\lambda(t,x) - \frac{2\lambda\beta}{T}\log M_\lambda.
\end{equation}
Equivalently,
\begin{equation}\label{direct_value_formula}
V_\lambda(t,x) = - \frac{2\lambda\beta}{T}\log \int_{\R^d} G_{T-t}(x-y)e^{-\alpha_\lambda f(y)}\dd y.
\end{equation}
\begin{remark}[Role of $T$ and $\beta$]\label{roleT}
Formula \eqref{direct_value_formula} makes the role of $T$ completely explicit. The Gaussian factor has covariance $2\beta(T-t)I_d$, so the typical radius of the terminal averaging at time $t$ is $\sqrt{2\beta(T-t)}$. This is the time-dependent search radius. At the same time, the Gibbs factor uses inverse temperature $\alpha_\lambda = \frac{T}{2\lambda\beta}$. Hence increasing $T$ simultaneously enlarges the spatial search region and sharpens the Gibbs concentration. With the present normalization, the limit $T \to +\infty$ is therefore not a neutral passage to a larger horizon but a different optimal control problem. An infinite-horizon theory would require a differently scaled criterion, typically discounted or ergodic. By a time rescaling one may normalize $\beta = 1$, but keeping $\beta$ explicit makes both scales visible.
\end{remark}

\begin{theorem}\label{main_value}
Under Assumption \ref{basic_ass}, the function $V_\lambda$ defined by \eqref{def_V_lambda} belongs to $C^{1,2}([0,T) \times \R^d) \cap C([0,T] \times \R^d)$ and is the classical solution of the backward viscous eikonal equation
\begin{equation}\label{hjb}
\partial_t V_\lambda = \frac{T}{2\lambda}\Vert \nabla V_\lambda \Vert^2 - \beta \Delta V_\lambda
\quad
\text{on } [0,T) \times \R^d,
\end{equation}
with terminal condition $V_\lambda(T,x) = f(x)$. Moreover,
$$
V_\lambda(t,x) = \inf_u J^\lambda_{t,x}(u),
$$
where the infimum is taken over all admissible controls. The optimal control is the Markov feedback
\begin{equation}\label{optimal_feedback_formula}
\mathbf u_\lambda^*(t,x) = - \frac{T}{\lambda}\nabla V_\lambda(t,x) = - \nabla \Phi_\lambda(t,x) = 2\beta \nabla \log h_\lambda(t,x).
\end{equation}
\end{theorem}
\noindent
The notation of $\mathbf u_\lambda^*$ as the optimal feedback will be used consistently below.

\begin{proof}
Since $\pi_\lambda \in C^2(\R^d) \cap L^1(\R^d)$, the heat convolution $h_\lambda$ belongs to $C^\infty([0,T) \times \R^d)$ and satisfies
\begin{equation}\label{h_lambda_heat}
\partial_t h_\lambda + \beta \Delta h_\lambda = 0
\qquad
\text{on } [0,T) \times \R^d,
\end{equation}
with terminal trace $h_\lambda(T,\cdot) = \pi_\lambda$. Hence
$$
\nabla V_\lambda = - \frac{2\lambda\beta}{T}\frac{\nabla h_\lambda}{h_\lambda},
\qquad
\Delta V_\lambda = - \frac{2\lambda\beta}{T}\frac{\Delta h_\lambda}{h_\lambda} + \frac{2\lambda\beta}{T}\frac{\Vert \nabla h_\lambda \Vert^2}{h_\lambda^2}.
$$
Using \eqref{h_lambda_heat}, we obtain \eqref{hjb}. Since $h_\lambda(t,\cdot) \to \pi_\lambda$ locally uniformly as $t \uparrow T$, the definition of $V_\lambda$ yields
$$
V_\lambda(T,x) = - \frac{2\lambda\beta}{T}\log \pi_\lambda(x) - \frac{2\lambda\beta}{T}\log M_\lambda = f(x).
$$
The Hamiltonian of the finite-horizon optimal control problem is (see, e.g., \cite{PontryaginBook, TrelatBook, YongZhou99})
$$
H(x,p,u) = u \cdot p+ \frac{\lambda}{2T}\Vert u \Vert^2 .
$$
For each $(x,p)$ it is minimized at $u = - \frac{T}{\lambda}p$, and the minimum value equals $- \frac{T}{2\lambda}\Vert p \Vert^2$. Therefore the HJB equation is \eqref{hjb}, and the minimizing selector is \eqref{optimal_feedback_formula}. Since $V_\lambda$ is a classical solution and the minimizer is explicit, the verification theorem for finite-horizon controlled diffusions applies (see \cite[Chapter IV]{FlemingSoner06} or \cite[Chapter 4]{YongZhou99}), hence $V_\lambda$ is the value function and \eqref{optimal_feedback_formula} is an optimal feedback.
\end{proof}

Theorem \ref{main_value} is classical in several neighboring literatures. The contribution of the present paper is not the existence of the closed form itself. What is specific here is the optimization reading of the same formula: the deterministic warm-up, the feedback interpretation through the Fokker-Planck equation and the Pontryagin maximum principle, the explicit conditional terminal law, the endogenous schedules set by the criterion, and the low-$\lambda$ asymptotics stated in a form directly relevant for global optimization.

\begin{remark}[Regularized potential]\label{regularized_potential}
The potential $\Phi_\lambda(t,\cdot) = - 2\beta \log \big( G_{T-t} * \pi_\lambda \big)$ is a heat-regularized Gibbs landscape. The Gaussian scale is $\sqrt{2\beta(T-t)}$, so earlier times correspond to broader smoothing. Proposition \ref{terminal_limit} further shows that this regularized landscape recovers $f$ as $t \uparrow T$. This is the natural connection with Gaussian homotopies and Hamilton-Jacobi regularizations, with the important difference that the present variance schedule is selected by \eqref{criterion_intro} rather than prescribed externally.
\end{remark}

\begin{remark}[Monotonicity in $\lambda$]\label{monotonicity_lambda}
The value function $V_\lambda(t,x)$ is nondecreasing in $\lambda$ for each fixed $(t,x)$. 
In particular, the two limiting values match the monotonicity: as $\lambda \downarrow 0$ the value function decreases to $f_*$ (Theorem \ref{small_lambda}), and as $\lambda \to +\infty$ the control cost forces $u \approx 0$, so $V_\lambda(t,x) \uparrow (G_{T-t} * f)(x)$, the free-diffusion average of $f$.
\end{remark}

\begin{remark}\label{pmp_remark}
The same feedback law also emerges from an optimal control problem posed directly on the density. Consider \eqref{FP_control} subject to \eqref{FP}. Applying the Pontryagin maximum principle (see \cite{LiYong95, TrelatBook}) on the infinite-dimensional state space of densities, with adjoint field $p(s,y)$, gives the adjoint equation
$$
- \partial_s p = \mathbf u \cdot \nabla p + \beta \Delta p + \frac{\lambda}{2T}\Vert \mathbf u \Vert^2,
\qquad
p(T,\cdot) = f,
$$
and pointwise minimization of the Hamiltonian yields $\mathbf u = - \frac{T}{\lambda}\nabla p$. Substituting back gives
$$
\partial_s p = \frac{T}{2\lambda}\Vert \nabla p \Vert^2 - \beta \Delta p,
$$
which is the backward Hamilton-Jacobi equation \eqref{hjb}. Above, we have used HJB and verification for the derivation in the present finite-dimensional diffusion setting, but the Pontryagin maximum principle could have been applied as well. Conceptually, it explains why the feedback formula is natural once the state variable is the whole density rather than a single trajectory.
\end{remark}

\subsection{Conditional terminal law and three equivalent formulas}

The next step is to identify the optimal Markov process and its terminal conditional law. This is where the optimization mechanism becomes transparent.

\medskip
\paragraph{\bf The optimal process as a Doob transform.}

\begin{lemma}\label{kernel_lemma}
For $0 \leq t < s \leq T$ define
$$
p_\lambda^*(t,x;s,y) = \frac{G_{s-t}(y-x) h_\lambda(s,y)}{h_\lambda(t,x)}.
$$
Then $p_\lambda^*$ is a Markov transition kernel. The associated Markov process has infinitesimal generator
$$
\mathcal L_t^* \varphi(x) = \beta \Delta \varphi(x) + \mathbf u_\lambda^*(t,x) \cdot \nabla \varphi(x),
$$
where $\mathbf u_\lambda^*$ is the feedback from Theorem \ref{main_value}. In particular, this process realizes the optimal controlled dynamics.
At terminal time, we have
\begin{equation}\label{terminal_eta}
p_\lambda^*(t,x;T,y) = \eta_{\lambda,t,x}(y),
\qquad
\eta_{\lambda,t,x}(y) = \frac{G_{T-t}(x-y)\pi_\lambda(y)}{h_\lambda(t,x)}.
\end{equation}
\end{lemma}

\begin{proof}
The positivity of $p_\lambda^*$ is obvious. Moreover,
\begin{multline*}
\int_{\R^d} p_\lambda^*(t,x;s,y)\dd y
= \frac{1}{h_\lambda(t,x)} \int_{\R^d} G_{s-t}(y-x)h_\lambda(s,y)\dd y \\
= \frac{1}{h_\lambda(t,x)} (G_{s-t} * h_\lambda(s,\cdot))(x)
= \frac{1}{h_\lambda(t,x)} (G_{T-t} * \pi_\lambda)(x)
= 1.
\end{multline*}
Similarly, for $t < r < s$,
\begin{multline*}
\int_{\R^d} p_\lambda^*(t,x;r,z)p_\lambda^*(r,z;s,y)\dd z
= \frac{h_\lambda(s,y)}{h_\lambda(t,x)} \int_{\R^d} G_{r-t}(z-x)G_{s-r}(y-z)\dd z \\
= \frac{G_{s-t}(y-x)h_\lambda(s,y)}{h_\lambda(t,x)}
= p_\lambda^*(t,x;s,y).
\end{multline*}
Hence $p_\lambda^*$ is a Markov kernel.

Let $P_{t,s}^*$ denote the corresponding semigroup:
$$
P_{t,s}^* \varphi(x) = \int_{\R^d} \varphi(y)p_\lambda^*(t,x;s,y)\dd y = \frac{1}{h_\lambda(t,x)} P_{s-t}\big( h_\lambda(s,\cdot)\varphi \big)(x),
$$
where $P_r = G_r * \cdot$ is the heat semigroup. Differentiating at $s = t$ and using \eqref{h_lambda_heat}, we obtain
\begin{multline*}
\lim_{s \downarrow t} \frac{P_{t,s}^* \varphi(x) - \varphi(x)}{s-t}
= \frac{1}{h_\lambda(t,x)} \Big( \beta \Delta \big( h_\lambda(t,\cdot)\varphi \big)(x) + \partial_t h_\lambda(t,x)\varphi(x) \Big) \\
= \beta \Delta \varphi(x) + 2\beta \frac{\nabla h_\lambda(t,x)}{h_\lambda(t,x)} \cdot \nabla \varphi(x)
= \beta \Delta \varphi(x) + \mathbf u_\lambda^*(t,x) \cdot \nabla \varphi(x).
\end{multline*}
This gives the generator. Setting $s = T$ yields \eqref{terminal_eta}.
\end{proof}

\medskip
\paragraph{\bf Penalized Gibbs interpretation of the conditional terminal law.}
The next proposition is the bridge with Proposition \ref{warmup}. It rewrites the conditional terminal law \eqref{terminal_eta} as a Gibbs law on a penalized energy.

\begin{proposition}\label{penalized_gibbs}
For every $0 \leq t < T$ and $x \in \R^d$, the density $\eta_{\lambda,t,x}$ from \eqref{terminal_eta} can be written as
\begin{equation}\label{penalized_eta}
\eta_{\lambda,t,x}(y) = \frac{1}{Z_{\lambda,t,x}} \exp \big( - \alpha_\lambda \mathcal E_{\lambda,t,x}(y) \big),
\end{equation}
where
\begin{equation}\label{penalized_energy}
\mathcal E_{\lambda,t,x}(y) = f(y) + \frac{\lambda}{2T(T-t)}\Vert y - x \Vert^2,
\end{equation}
and $Z_{\lambda,t,x}$ is the normalizing constant
$$
Z_{\lambda,t,x} = (4\pi \beta (T-t))^{-d/2}\int_{\R^d} \exp \big( - \alpha_\lambda \mathcal E_{\lambda,t,x}(z) \big)\dd z.
$$
\end{proposition}

\begin{proof}
Using the definitions of $G_{T-t}$, $\pi_\lambda$, and $\alpha_\lambda$,
\begin{align*}
G_{T-t}(x-y)\pi_\lambda(y)
&= \frac{1}{(4\pi \beta (T-t))^{d/2}M_\lambda} \exp \bigg( - \frac{\Vert x - y \Vert^2}{4\beta(T-t)} - \alpha_\lambda f(y) \bigg) \\
&= \frac{1}{(4\pi \beta (T-t))^{d/2}M_\lambda} \exp \bigg( - \alpha_\lambda \Big( f(y) + \frac{\lambda}{2T(T-t)}\Vert x - y \Vert^2 \Big) \bigg).
\end{align*}
Dividing by $h_\lambda(t,x)$ yields \eqref{penalized_eta}.
\end{proof}

Proposition \ref{penalized_gibbs} should be read literally. At the current state $(t,x)$, each candidate terminal point $y$ receives a weight that rewards low objective values through $f(y)$ and penalizes kinematic distance from the current position through the quadratic term $\displaystyle \frac{\lambda}{2T(T-t)}\Vert y - x \Vert^2$. At the initial time $t = 0$, the penalized energy in \eqref{penalized_energy} is 
$$
y \mapsto f(y) + \frac{\lambda}{2T^2}\Vert y - x \Vert^2,
$$
namely the same proximal energy as in Proposition \ref{warmup}. In that sense the stochastic problem is a Gibbs relaxation of the deterministic warm-up.

\medskip

\paragraph{\bf Three equivalent formulas for the optimal drift.}
We next give three representations of the optimal feedback. The potential form is the Hamilton-Jacobi representation; the averaged-gradient and barycentric forms are the two conditional-expectation representations. The differential and one-sided Lipschitz structure of the barycenter is recorded separately afterwards, because it is a consequence of the barycentric representation rather than a fourth formula for the drift.

\begin{proposition}[Three representations of the optimal drift]\label{three_forms}
Under Assumption \ref{basic_ass}, the optimal feedback admits the following equivalent representations on $[0,T)\times\R^d$. First, it is induced by the potential
\begin{equation}\label{potential_form}
\mathbf u_\lambda^*(t,x) = - \nabla \Phi_\lambda(t,x), \qquad 
\Phi_\lambda(t,x) = -2\beta \log (G_{T-t} * \pi_\lambda)(x).
\end{equation}
Second, it has the averaged-gradient form
\begin{equation}\label{gradient_and_barycenter}
\mathbf u_\lambda^*(t,x) = - \frac{T}{\lambda}\int_{\R^d} \nabla f(y)\,\eta_{\lambda,t,x}(y)\dd y.
\end{equation}
Third, it has the barycentric form
\begin{equation}\label{def_a_lambda}
\mathbf u_\lambda^*(t,x) = - \frac{x - a_\lambda(t,x)}{T-t}, \qquad a_\lambda(t,x) = \int_{\R^d} y\, \eta_{\lambda,t,x}(y)\dd y.
\end{equation}
\end{proposition}

\begin{proof}
The potential form is just the definition of $\Phi_\lambda$ together with $\mathbf u_\lambda^* = 2\beta \nabla_x \log h_\lambda$ and $h_\lambda = G_{T-t}*\pi_\lambda$. For the averaged-gradient form, we integrate by parts:
\begin{equation*}
\begin{split}
\nabla_x h_\lambda(t,x)
&=\int_{\R^d} \nabla_x G_{T-t}(x-y)\pi_\lambda(y)\dd y\\
&=- \int_{\R^d} \nabla_y G_{T-t}(x-y)\pi_\lambda(y)\dd y 
=\int_{\R^d} G_{T-t}(x-y)\nabla_y \pi_\lambda(y)\dd y.
\end{split}
\end{equation*}
Since $\nabla \pi_\lambda = - \alpha_\lambda \pi_\lambda \nabla f$, we obtain
$$
\nabla_x h_\lambda(t,x) = - \alpha_\lambda \int_{\R^d} G_{T-t}(x-y)\pi_\lambda(y)\nabla f(y)\dd y.
$$
Multiplying by $\frac{2\beta}{h_\lambda(t,x)}$ and using $2\beta\alpha_\lambda = \frac{T}{\lambda}$ yields \eqref{gradient_and_barycenter}. Finally,
$$
\nabla_x h_\lambda(t,x)
= \int_{\R^d} \nabla_x G_{T-t}(x-y)\pi_\lambda(y)\dd y
= - \frac{1}{2\beta(T-t)} \int_{\R^d} (x-y)G_{T-t}(x-y)\pi_\lambda(y)\dd y.
$$
Hence
$$
\mathbf u_\lambda^*(t,x) = 2\beta \frac{\nabla_x h_\lambda(t,x)}{h_\lambda(t,x)} = - \frac{x - a_\lambda(t,x)}{T-t},
$$
which proves \eqref{def_a_lambda}.
\end{proof}

\paragraph{\bf Differential structure of the barycenter.}
We now turn to the monotonicity and one-sided Lipschitz properties of $a_\lambda(t,\cdot)$. If $\eta$ is a probability density on $\R^d$, we denote by $Y$ the canonical random variable with law $\eta(y)\dd y$, by
$$
\E_\eta[\psi(Y)] = \int_{\R^d}\psi(y)\eta(y)\dd y,
\qquad m_\eta = \E_\eta[Y],
$$
and by
$$
\operatorname{Cov}_\eta(Y)
= \E_\eta\big[(Y-m_\eta)(Y-m_\eta)^\top\big]
= \int_{\R^d}(y-m_\eta)(y-m_\eta)^\top\eta(y)\dd y
$$
the covariance matrix. For $\xi\in\R^d$,
$$
\operatorname{Var}_\eta(\xi\cdot Y)
= \E_\eta\big[(\xi\cdot(Y-m_\eta))^2\big]
= \xi^\top\operatorname{Cov}_\eta(Y)\xi.
$$
When $\eta=\eta_{\lambda,t,x}$, this convention gives $m_\eta=a_\lambda(t,x)$.

\begin{proposition}[Barycentric monotonicity and covariance-Hessian identity]\label{barycenter_differential}
Under Assumption \ref{basic_ass}, $a_\lambda(t,\cdot)$ is $C^1$ and satisfies the matrix identity
\begin{equation}\label{Ceta_factorization}
D_xa_\lambda(t,x)
= I_d-(T-t)D_x^2\Phi_\lambda(t,x) 
= \frac{1}{2\beta(T-t)}\operatorname{Cov}_{\eta_{\lambda,t,x}}(Y).
\end{equation}
Here the covariance is computed under the law $\eta_{\lambda,t,x}(y)\dd y$. In particular, $D_xa_\lambda(t,x)$ is symmetric nonnegative definite and $a_\lambda(t,\cdot)$ is monotone.
\end{proposition}

\begin{proof}
Fix $(t,x)$ and let $Y$ have law $\eta_{\lambda,t,x}(y)\dd y$. Writing $\log \eta_{\lambda,t,x}(y) = -\frac{\Vert x-y\Vert^2}{4\beta(T-t)} + \log\pi_\lambda(y) - \log h_\lambda(t,x)$ and using $\nabla_x \log h_\lambda(t,x) = -\frac{1}{2\beta}\nabla\Phi_\lambda(t,x) = \frac{a_\lambda(t,x) - x}{2\beta(T-t)}$, we obtain
$$
\nabla_x\log \eta_{\lambda,t,x}(y)
= \frac{y-a_\lambda(t,x)}{2\beta(T-t)}.
$$
Therefore, for every direction $\xi\in\R^d$,
\begin{multline*}
D_xa_\lambda(t,x)\xi
= \int_{\R^d} y\,\xi\cdot\nabla_x\log\eta_{\lambda,t,x}(y)\,\eta_{\lambda,t,x}(y)\dd y \\
= \frac{1}{2\beta(T-t)}\int_{\R^d} y\,\xi\cdot\big(y-a_\lambda(t,x)\big)\eta_{\lambda,t,x}(y)\dd y 
= \frac{1}{2\beta(T-t)}\operatorname{Cov}_{\eta_{\lambda,t,x}}(Y)\xi.
\end{multline*}
This proves the covariance identity. On the other hand, the barycentric formula and the potential representation give $a_\lambda(t,x)=x-(T-t)\nabla\Phi_\lambda(t,x)$, hence $D_xa_\lambda(t,x)=I_d-(T-t)D_x^2\Phi_\lambda(t,x)$. Since covariance matrices are nonnegative definite, $D_xa_\lambda(t,x)$ is nonnegative definite. The monotonicity follows by integrating $D_xa_\lambda$ along line segments.
\end{proof}

\paragraph{\bf Factorized one-sided Lipschitz identities.}
Let $t<T$, $x_1,x_2\in\R^d$, and $x(\theta)=x_1+\theta(x_2-x_1)$. We define the trace covariance along this segment by
\begin{equation}\label{eq:Ceta}
C_{\eta_\lambda}(t,x_1,x_2)
= \int_0^1 \operatorname{tr}\operatorname{Cov}_{\eta_{\lambda,t,x(\theta)}}(Y)\dd\theta 
= \int_0^1\int_{\R^d}\Vert y-a_\lambda(t,x(\theta))\Vert^2\eta_{\lambda,t,x(\theta)}(y)\dd y\dd\theta.
\end{equation}
Taking the trace in \eqref{Ceta_factorization} gives, for every $x$,
$$
\operatorname{tr}\operatorname{Cov}_{\eta_{\lambda,t,x}}(Y)
= 2\beta(T-t)\big(d-(T-t)\Delta\Phi_\lambda(t,x)\big).
$$
After integration along the segment, this yields the exact factorization
\begin{equation}\label{Ceta_terminal_factorized}
C_{\eta_\lambda}(t,x_1,x_2)
= 2\beta(T-t)\int_0^1\big(d-(T-t)\Delta\Phi_\lambda(t,x(\theta))\big)\dd\theta.
\end{equation}

\begin{corollary}[Directional, variance and trace forms of the one-sided Lipschitz bound]\label{factorized_osl_cor}
Let $t<T$ and $x_1,x_2\in \R^d$. With $x(\theta)=x_1+\theta(x_2-x_1)$, the directional identity
\begin{equation}\label{eq:OSL_directional}
\big\langle a_\lambda(t,x_2)-a_\lambda(t,x_1),x_2-x_1\big\rangle 
= \int_0^1 (x_2-x_1)^\top \big(I_d-(T-t)D_x^2\Phi_\lambda(t,x(\theta))\big)(x_2-x_1)\dd\theta
\end{equation}
holds. Equivalently, in covariance form,
\begin{equation}\label{eq:OSL_variance}
\big\langle a_\lambda(t,x_2)-a_\lambda(t,x_1),x_2-x_1\big\rangle
= \frac{1}{2\beta(T-t)}\int_0^1 \operatorname{Var}_{\eta_{\lambda,t,x(\theta)}}\big((x_2-x_1)\cdot Y\big)\dd\theta,
\end{equation}
where, for each $\theta$, the random variable $Y$ has law $\eta_{\lambda,t,x(\theta)}(y)\dd y$. Moreover,
\begin{equation}\label{eq:OSLC}
0 \leq \big\langle a_\lambda(t,x_2)-a_\lambda(t,x_1),x_2-x_1\big\rangle
\leq \Vert x_2-x_1\Vert^2\int_0^1\big(d-(T-t)\Delta\Phi_\lambda(t,x(\theta))\big)\dd\theta.
\end{equation}
\end{corollary}

\begin{proof}
Integrating \eqref{Ceta_factorization} along the segment from $x_1$ to $x_2$ gives
$$
 a_\lambda(t,x_2)-a_\lambda(t,x_1) = \int_0^1D_xa_\lambda(t,x(\theta))(x_2-x_1)\dd\theta.
$$
Taking the scalar product with $x_2-x_1$ and using the identity
$D_xa_\lambda(t,x)=I_d-(T-t)D_x^2\Phi_\lambda(t,x)$
gives \eqref{eq:OSL_directional}. Using instead
$D_xa_\lambda(t,x)=\frac{1}{2\beta(T-t)}\operatorname{Cov}_{\eta_{\lambda,t,x}}(Y)$
gives \eqref{eq:OSL_variance}, because
$$
(x_2-x_1)^\top\operatorname{Cov}_{\eta_{\lambda,t,x(\theta)}}(Y)(x_2-x_1)
= \operatorname{Var}_{\eta_{\lambda,t,x(\theta)}}\big((x_2-x_1)\cdot Y\big).
$$
It remains to prove the trace bound. For every $\theta$, the matrix
$$
A_\theta=I_d-(T-t)D_x^2\Phi_\lambda(t,x(\theta))
= \frac{1}{2\beta(T-t)}\operatorname{Cov}_{\eta_{\lambda,t,x(\theta)}}(Y)
$$
is symmetric nonnegative definite. Therefore, for every $z\in\R^d$,
$z^\top A_\theta z \leq \Vert z\Vert^2\operatorname{tr} A_\theta$.
Applying this with $z=x_2-x_1$ and using
$\operatorname{tr}A_\theta=d-(T-t)\Delta\Phi_\lambda(t,x(\theta))$
yields \eqref{eq:OSLC} after integration in $\theta$.
\end{proof}

\begin{remark}[Near-terminal behavior and Hamilton-Jacobi interpretation]\label{terminal_semiconcavity_rem}
We spell out the near-terminal content of the preceding identities. As $t\uparrow T$, the conditional law $\eta_{\lambda,t,x}$ localizes near $x$ at scale $\sqrt{T-t}$. Under Assumption \ref{basic_ass}, $\pi_\lambda$ is $C^2$ on $\R^d$ with locally bounded derivatives. The Gaussian semigroup $G_\tau\,*$ is a smooth approximation of the identity, so $G_{T-t}*\pi_\lambda \to \pi_\lambda$ in $C^2_{\mathrm{loc}}(\R^d)$ as $t\uparrow T$. Since $\pi_\lambda$ is positive and the logarithm is smooth on positive functions bounded away from zero on compact sets, this convergence also holds after applying $-2\beta\log$. Therefore
\begin{equation}\label{terminal_hessian_phi}
D_x^2\Phi_\lambda(t,\cdot)
\longrightarrow -2\beta D_x^2\log\pi_\lambda = \frac{T}{\lambda}D^2f
\qquad\hbox{locally uniformly as } t\uparrow T.
\end{equation}
Inserting \eqref{terminal_hessian_phi} into \eqref{eq:OSL_directional} gives, uniformly for $x_1,x_2$ in compact sets,
\begin{multline}\label{terminal_directional_osl}
\big\langle a_\lambda(t,x_2)-a_\lambda(t,x_1),x_2-x_1\big\rangle \\
= \Vert x_2-x_1\Vert^2 - \frac{T(T-t)}{\lambda}\int_0^1 (x_2-x_1)^\top D^2f(x(\theta))(x_2-x_1)\dd\theta 
+ \Vert x_2-x_1\Vert^2\mathrm{o}(T-t).
\end{multline}
Thus the sharp directional coefficient tends to $1$, consistently with the terminal identity $a_\lambda(T,x)=x$. Similarly, \eqref{Ceta_terminal_factorized} yields the trace expansion
\begin{equation}\label{Ceta_terminal_expansion}
C_{\eta_\lambda}(t,x_1,x_2)
= 2d\beta(T-t) - \frac{2\beta T}{\lambda}(T-t)^2\int_0^1\Delta f(x(\theta))\dd\theta + \mathrm{o}\big((T-t)^2\big).
\end{equation}
Thus the directional coefficient in \eqref{eq:OSL_directional} converges to $1$, whereas the trace majorant in \eqref{eq:OSLC} converges to $d$. The trace bound is less sharp, but it is useful when only the scalar trace covariance $C_{\eta_\lambda}$ is tracked.

The Hamilton-Jacobi interpretation is as follows. From \eqref{Ceta_factorization},
$$
I_d-(T-t)D_x^2\Phi_\lambda(t,x)\geq 0,
\qquad
D_x^2\Phi_\lambda(t,x)\leq \frac{1}{T-t}I_d.
$$
This is the usual one-sided semiconcavity barrier for a quadratic Hamilton-Jacobi flow; see, for example, \cite{BardiCapuzzoDolcetta97, CannarsaSinestrari04, lions1982generalized}. It is an upper barrier, not an assertion that the Hessian actually blows up. Indeed, \eqref{Ceta_factorization} is equivalently
$$
D_x^2\Phi_\lambda(t,x) = \frac{1}{T-t}I_d - \frac{1}{2\beta(T-t)^2}\operatorname{Cov}_{\eta_{\lambda,t,x}}(Y).
$$
Combining this identity with \eqref{terminal_hessian_phi} gives, componentwise and locally uniformly in $x$,
$$
\operatorname{Cov}_{\eta_{\lambda,t,x}}(Y_i,Y_j)
= 2\beta(T-t)\delta_{ij} - \frac{2\beta T}{\lambda}(T-t)^2\partial_{ij}f(x) + \mathrm{o}\big((T-t)^2\big).
$$
Thus the diagonal covariance contains the leading heat-kernel term $2\beta(T-t)$, which cancels the diagonal semiconcavity barrier $\frac{1}{T-t}I_d$ in the Hessian formula; the off-diagonal covariance is only of order $(T-t)^2$, so no off-diagonal singularity is present.

From the geometric viewpoint of optimal control, the inviscid analogue would correspond to a Riccati equation for the Hessian along characteristics, whose finite-time blow-up marks the formation of caustics or conjugate points (see, e.g., \cite{AgrachevSachkov04,BonnardCaillauTrelat07}). In the present viscous whole-space setting, the Cole-Hopf representation together with the covariance identity \eqref{Ceta_factorization} shows directly that no such terminal singularity occurs.
\end{remark}

\begin{corollary}[Conditional terminal expectations]\label{conditional_cor}
Let $X^*$ denote the optimal Markov process from Lemma \ref{kernel_lemma}. Then, conditionally on $X_t^* = x$, the terminal law of $X_T^*$ is $\eta_{\lambda,t,x}$. Consequently,
$$
a_\lambda(t,x) = \E \big[ X_T^* \,\vert\, X_t^* = x \big],
$$
and
$$
\mathbf u_\lambda^*(t,x) = - \frac{T}{\lambda}\E \big[ \nabla f(X_T^*) \,\vert\, X_t^* = x \big] = - \frac{x - \E \big[ X_T^* \,\vert\, X_t^* = x \big]}{T-t}.
$$
\end{corollary}

\begin{proof}
The first statement is \eqref{terminal_eta}. The formulas then follow from Proposition \ref{three_forms}.
\end{proof}

Corollary \ref{conditional_cor} is the main message of the section. At time $t$ and current position $x$, the optimal drift is determined by a weighted family of candidate terminal points. The three formulas say the same thing in three different languages:
\begin{enumerate}[label=\arabic*)]
\item {\bf Potential language}. The optimal feedback is dictated as a heat/Hamilton-Jacobi regularization of the terminal landscape. Namely,  $\mathbf u_\lambda^* = - \nabla \Phi_\lambda$, where $\Phi_\lambda$ is the negative logarithm of a heat-regularized Gibbs density, or equivalently, $\mathbf u_\lambda^*$ is the negative gradient of $\frac{T}{\lambda}V_\lambda$, up to an irrelevant additive constant in the potential.  As $t \uparrow T$, Proposition \ref{terminal_limit} gives $V_\lambda(t,\cdot) \to f$ locally uniformly; equivalently, $\frac{\lambda}{T}\Phi_\lambda(t,\cdot) - \frac{2\lambda\beta}{T}\log M_\lambda \to f$ locally uniformly.

\begin{figure}[!htb]
	\centering
	\subcaptionbox{\label{subfig:t=0.1}}{
		\includegraphics[width=0.25\textwidth]{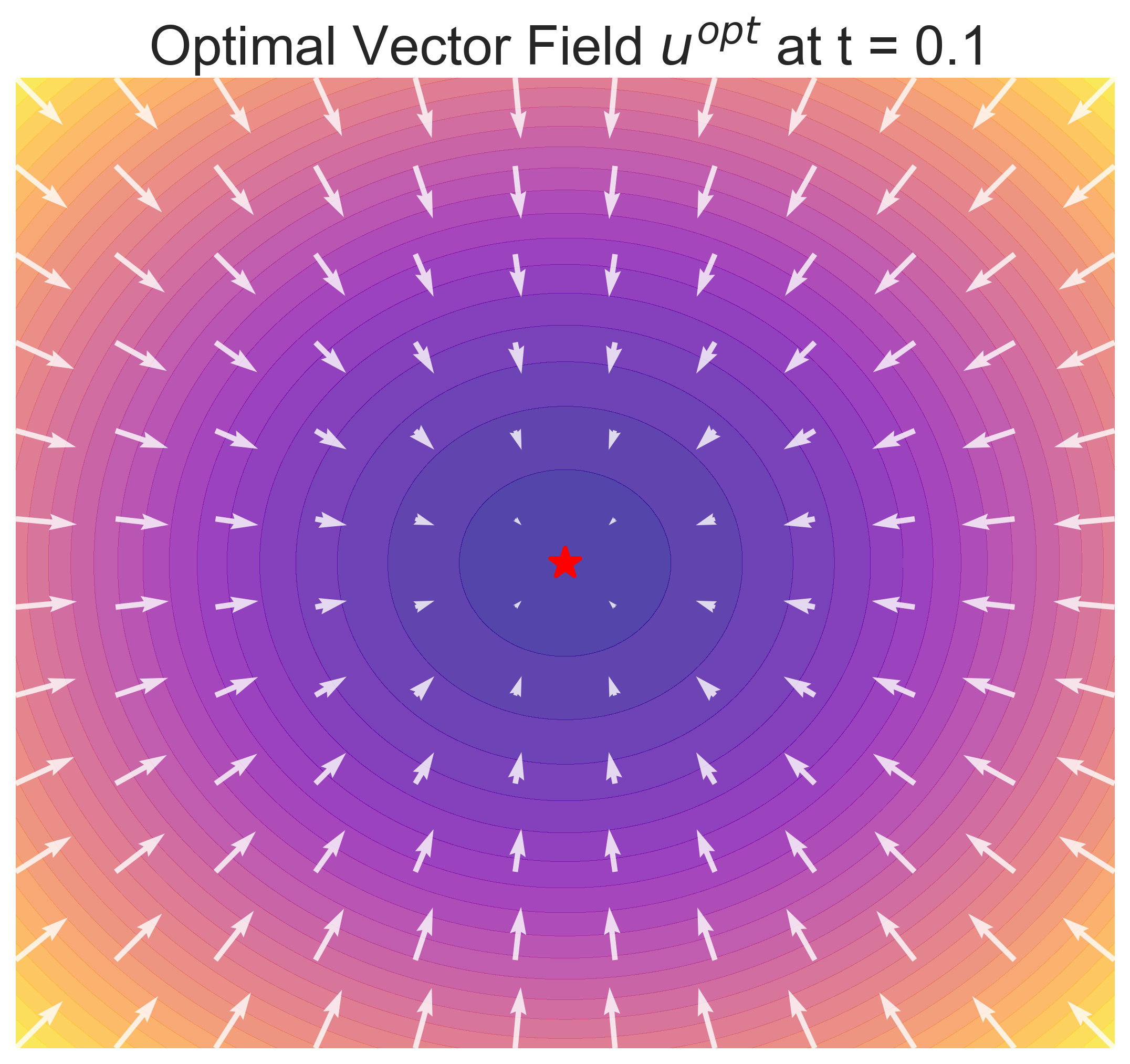}
	} 
	\subcaptionbox{\label{subfig:t=0.9}}{
		\includegraphics[width=0.25\textwidth]{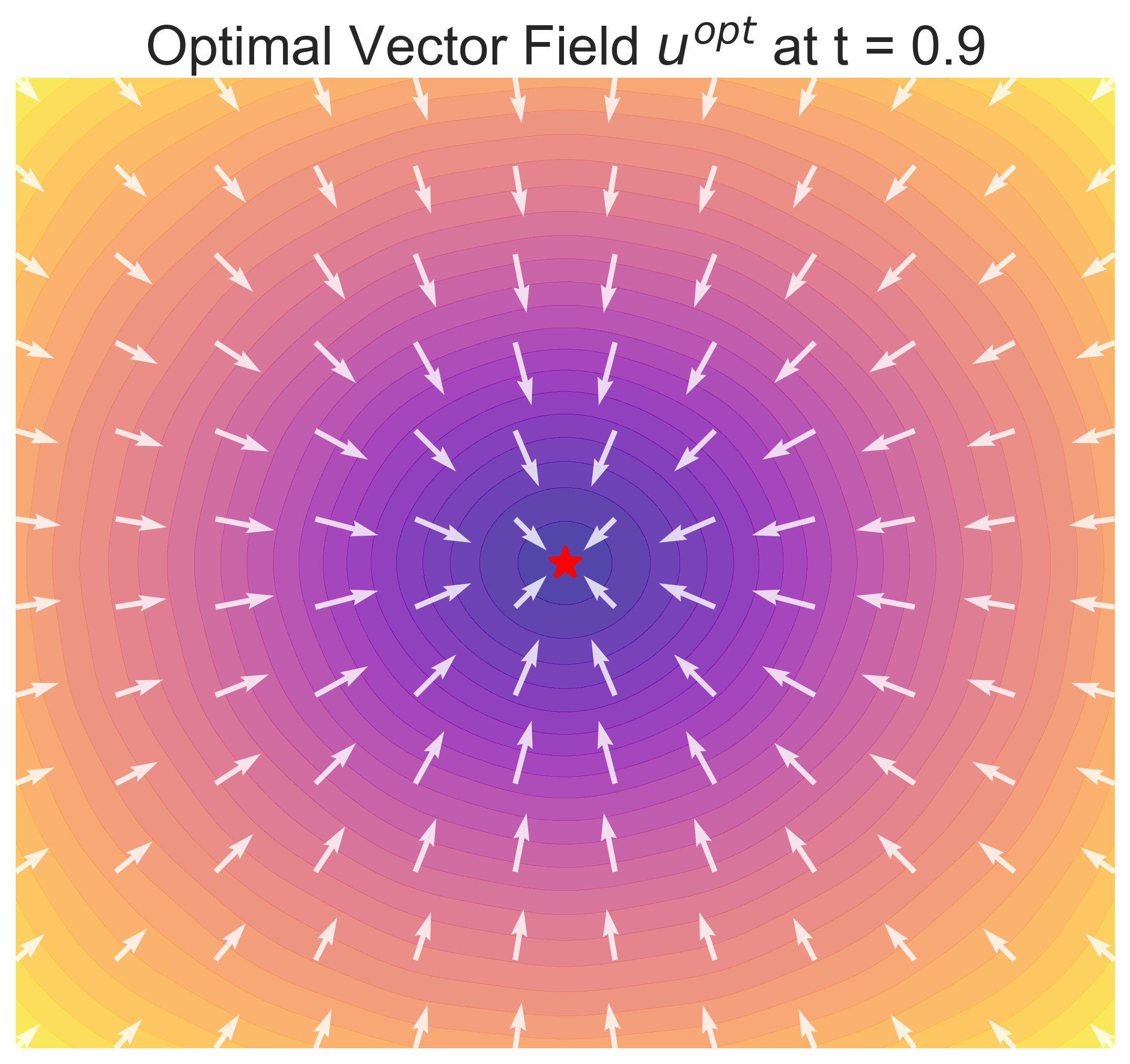}
	} 
	\caption{Visualization of $\Phi_\lambda$ and the control $\mathbf u_\lambda^* = -\nabla_x \Phi_\lambda(t,x)$ at different time $t$.
    We choose $f$ to be the 2-dimensional Ackley function.
    The parameters are set as $\lambda = 1$.
    The color-contour represents the landscape of $\Phi_\lambda(t, x)$ defined in \eqref{direct_value_formula}.
    The red star marks the target minimizer $x^*$ of $f$.
    The white arrows indicate the vector field. 
    The sequence of plots suggests that $\Phi_\lambda(t,x)$ is a strongly smoothed effective landscape when $t$ is small, and that this smoothing decreases as $t$ approaches $T=1$.
    }
	\label{fig:vf_vis}
\end{figure}

\item {\bf Averaged-gradient language}. The drift is the Gibbs-Gaussian average of the local gradients $\nabla f(y)$.
This weighted form of $\mathbf u_\lambda^*$ involves the normalized weight
$$
\eta_{\lambda,t,x}(y) \propto \exp \bigg( - \frac{T}{2\lambda\beta}f(y) - \frac{\Vert y - x \Vert^2}{4\beta(T-t)} \bigg).
$$
The first exponential rewards small values of the objective. The second rewards terminal points compatible with the current position and the remaining time. 
As $t \uparrow T$, the optimal feedback recovers a scaled negative gradient field; specifically, Proposition \ref{terminal_limit} gives
$$
\lim_{t \uparrow T} \mathbf u_\lambda^*(t,x) = -\frac{T}{\lambda}\nabla f(x),
$$
locally uniformly in $x$. In the normalization $T=1$, this becomes $-\frac{1}{\lambda}\nabla f(x)$. This alignment near $T=1$ is illustrated in Figure \ref{fig:vf_vis_compared}.

\begin{figure}[!htb]
	\centering
	\subcaptionbox{\label{subfig:t=0.99}}{
		\includegraphics[width=0.25\textwidth]{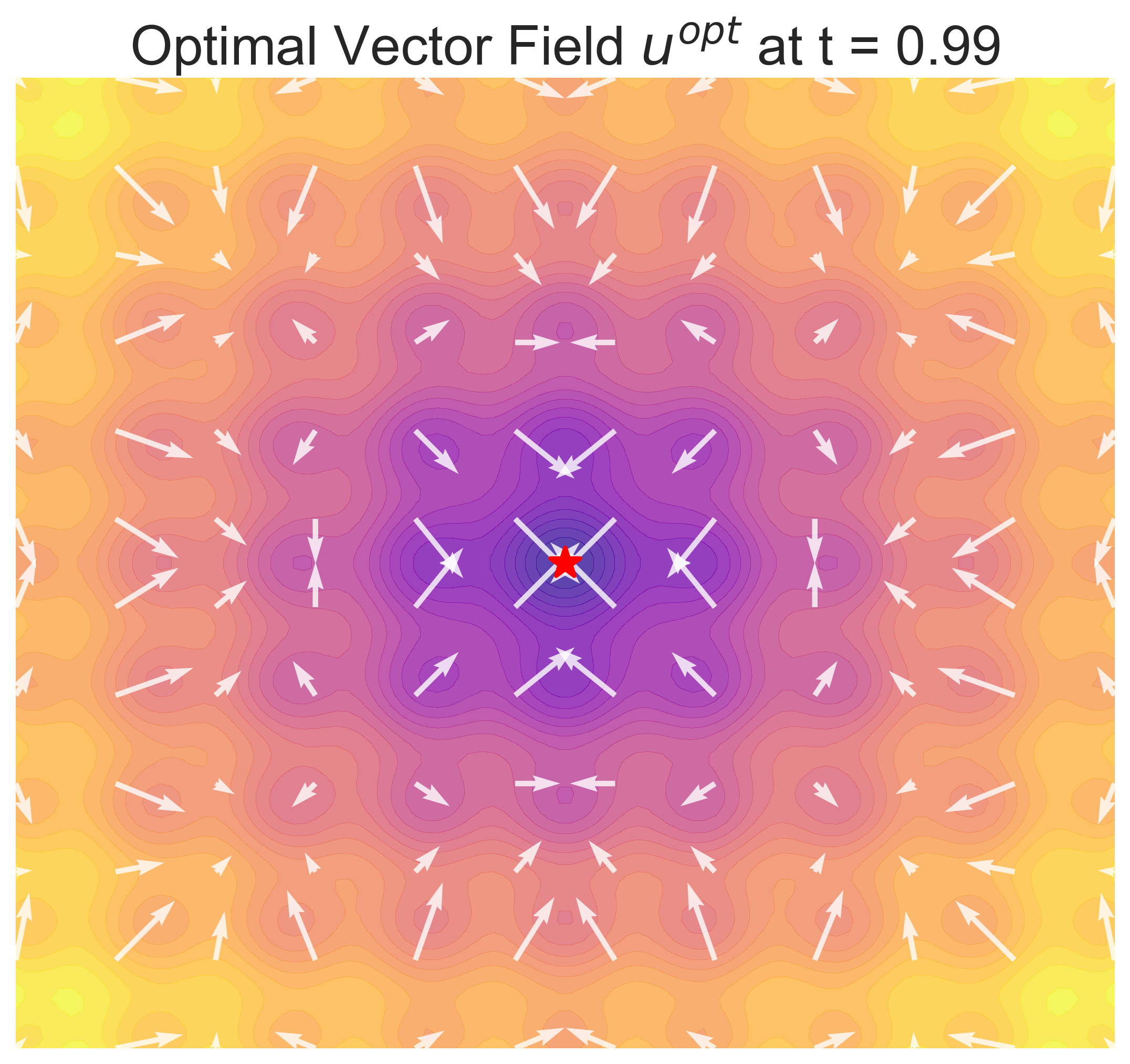}
	} 
	\subcaptionbox{\label{subfig:neg_grad}}{
		\includegraphics[width=0.25\textwidth]{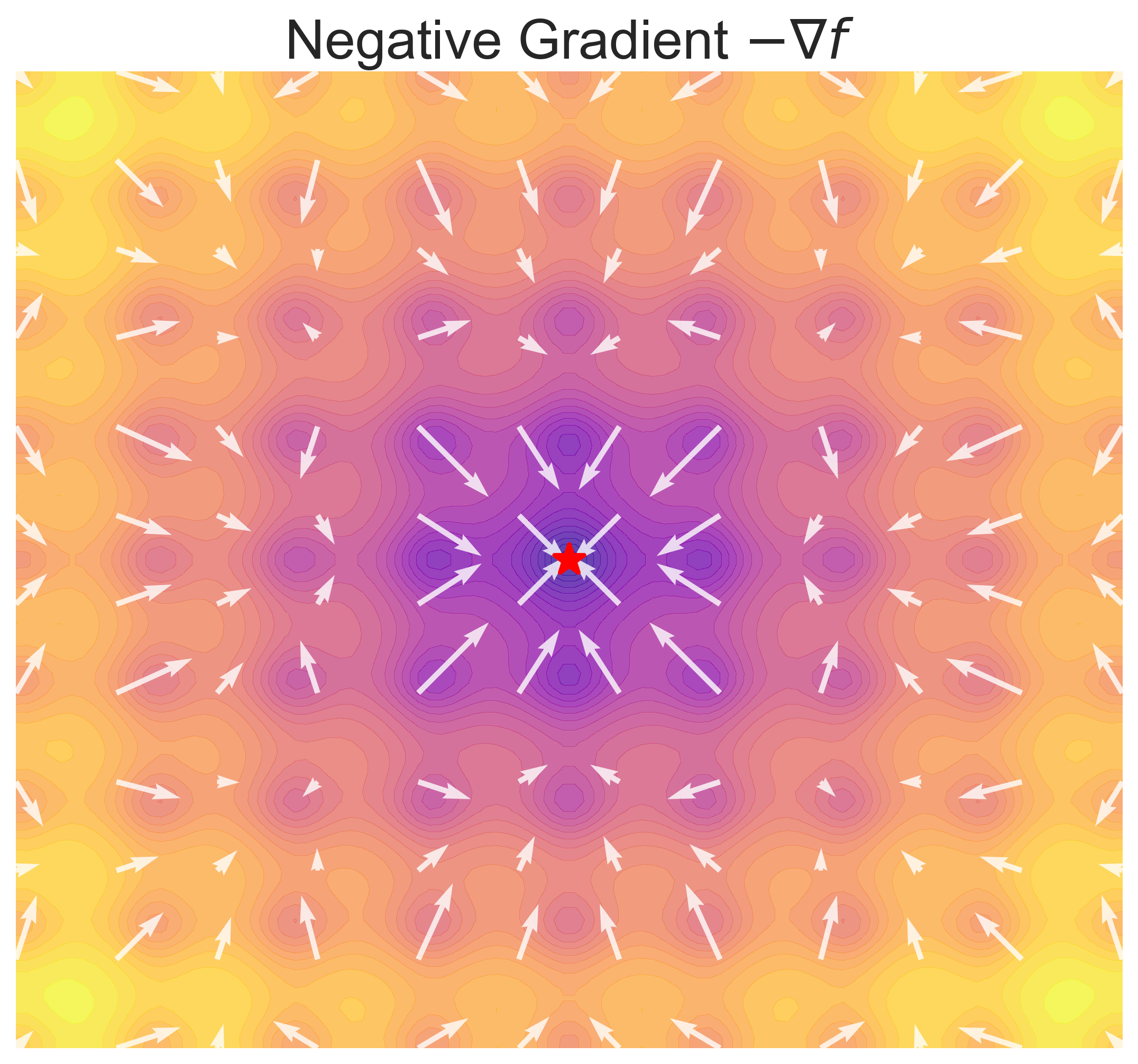}
	}
	\caption{Visualization of $(\Phi_\lambda, \mathbf u_\lambda^*)$ for $t = 0.99$ and $(f, -\nabla f)$.
         They are visually almost indistinguishable. Here $f$ is the two-dimensional Ackley function.
      }
	\label{fig:vf_vis_compared}
\end{figure}   

\item {\bf Barycentric language}. The drift points from the current state $x$ toward the weighted average terminal point $a_\lambda(t,x)$.
The barycentric formula then says: compute the weighted average of all candidate terminal points and move from $x$ toward that average.
It is this gradient-free formulation of $\mathbf u_\lambda^*$ that leads to the computational discretization recorded in Section \ref{numerics_sec} below.
The barycentric form also carries the factorized OSL and semiconcavity structure of $a_\lambda(t,\cdot)$ described in Proposition \ref{barycenter_differential} and Corollary \ref{factorized_osl_cor}.
\end{enumerate}

Two consequences are immediate. First, the drift may be evaluated either from function values alone through the barycentric formula or, when gradients are available, through the averaged-gradient formula. Second, unlike in Gaussian homotopy or consensus-based optimization, the variance $2\beta(T-t)$ and the effective temperature $\displaystyle \frac{2\lambda\beta}{T}$ are not chosen manually: they are selected by the control criterion itself.

\begin{corollary}[Exact optimal initialization]\label{exact_init}
If $X_0$ has density $h_\lambda(0,\cdot)$ and evolves under the optimal feedback $\mathbf u_\lambda^*$, then the law of $X_s$ is $h_\lambda(s,\cdot)$ for every $s \in [0,T]$. In particular,
$\Law(X_T) = \pi_\lambda(y)\dd y$.
\end{corollary}

\begin{proof}
For $s \in [0,T]$, the law at time $s$ is
$$
\int_{\R^d} p_\lambda^*(0,x;s,\cdot)h_\lambda(0,x)\dd x
= h_\lambda(s,\cdot)\int_{\R^d} G_s(\cdot - x)\dd x
= h_\lambda(s,\cdot),
$$
by the definition of $p_\lambda^*$ and the heat semigroup property.
\end{proof}

\begin{remark}[Path-space reweighting]\label{path_space_remark}
There is also an equivalent path-space formulation behind the same formulas. Let $\W_{t,x}$ be Wiener measure on the trajectory space $C([t,T],\R^d)$ associated with the driftless diffusion $X_s = x + \sqrt{2\beta}(B_s - B_t)$. Then the value function may be rewritten as
$$
V_\lambda(t,x)
=
\inf \left\{ \E_{\Qbb}[f(X_T)] + \frac{2\lambda\beta}{T}\Ent(\Qbb \vert \W_{t,x}) \ \ \big\vert\ \ \Qbb \in \Prob\big( C([t,T],\R^d) \big) \right\},
$$
and the minimizer is the exponentially reweighted Brownian law
$$
\frac{\dd \Qbb_{\lambda,t,x}^*}{\dd \W_{t,x}}
=
\frac{\exp \big( - \alpha_\lambda f(X_T) \big)}{\E_{\W_{t,x}}\big[ \exp \big( - \alpha_\lambda f(X_T) \big) \big]}.
$$
This means that Brownian trajectories are reweighted according to their terminal value: paths ending at smaller values of $f(X_T)$ receive larger weight. This is the bridge with the reciprocal-diffusion and Schr\"odinger-F\"ollmer literature \cite{DaiJiaoKangLuYang23, DaiPra91, Follmer85, HuangJiaoKangLiaoLiuLiu21, TzenRaginsky19, ZhangChen22}.
\end{remark}

\section{Asymptotics and optimization meaning}\label{asymp_sec}

This is the main optimization section. The low-$\lambda$ regime is the global regime: the effective temperature $\frac{2\lambda\beta}{T}$ tends to $0$, the terminal Gibbs law concentrates on the global minimizer, and the drift becomes an affine attraction toward it. The regime $t \to T$ is local: it shows that the same finite-horizon drift reconnects with a scaled gradient descent near terminal time. We present the local regime first because its proof is immediate, then the low-temperature regime, and finally the nondegenerate Laplace expansion.

\subsection{Near-terminal regime}

The limit $t \to T$ explains the local meaning of the finite-horizon construction. The Gaussian factor in \eqref{terminal_eta} shrinks to a Dirac mass, so the optimizer stops averaging over distant candidate terminal points and recovers local first-order information.

\begin{proposition}\label{terminal_limit}
Fix $\lambda > 0$. Under Assumption \ref{basic_ass}, we have
$$
V_\lambda(t,x) \to f(x),
\qquad
\mathbf u_\lambda^*(t,x) \to - \frac{T}{\lambda}\nabla f(x),
$$
as $t \uparrow T$, locally uniformly in $x$.
\end{proposition}

\begin{proof}
Since $h_\lambda(t,\cdot) = G_{T-t} * \pi_\lambda$, the family $h_\lambda(t,\cdot)$ converges locally uniformly to $\pi_\lambda$ as $t \uparrow T$. Because $\nabla \pi_\lambda \in L^1(\R^d)$ by Assumption \ref{basic_ass}, we also have
$\nabla h_\lambda(t,\cdot) = G_{T-t} * \nabla \pi_\lambda \to \nabla \pi_\lambda$
locally uniformly. Hence
$$
V_\lambda(t,x)
=
- \frac{2\lambda\beta}{T}\log h_\lambda(t,x) - \frac{2\lambda\beta}{T}\log M_\lambda
\longrightarrow
- \frac{2\lambda\beta}{T}\log \pi_\lambda(x) - \frac{2\lambda\beta}{T}\log M_\lambda
= f(x),
$$
locally uniformly. Likewise,
$$
\mathbf u_\lambda^*(t,x) = 2\beta \frac{\nabla h_\lambda(t,x)}{h_\lambda(t,x)} 
\longrightarrow
2\beta \frac{\nabla \pi_\lambda(x)}{\pi_\lambda(x)} = - \frac{T}{\lambda}\nabla f(x),
$$
locally uniformly.
\end{proof}

Thus, for fixed $\lambda$, the finite-horizon drift always recovers a scaled local gradient descent direction at the end of the interval. The nonlocal part of the optimizer is therefore concentrated away from terminal time.

\subsection{Low-temperature regime}

We now turn to the main global-optimization regime. The effective temperature of the terminal Gibbs law is $\displaystyle \frac{2\lambda\beta}{T}$, so the limit $\lambda \downarrow 0$ is a low-temperature limit. This is the regime in which the terminal law collapses on the global minimizer.

\begin{assumption}\label{unique_min}
The function $f$ has a unique minimizer $x^* \in \R^d$.
\end{assumption}

\begin{remark}[Local minima are allowed]\label{local_min_remark}
Assumption \ref{unique_min} requires only that the global minimizer be unique. The function $f$ may still possess arbitrarily many nonglobal local minima and saddle points. The results below therefore describe a global selection mechanism for the exact finite-horizon optimal control criterion.
\end{remark}

The first ingredient is the Laplace principle (see, for example, \cite{Hwang80,DemboZeitouni}) for the partition function $M_\lambda$.

\begin{lemma}\label{partition_laplace}
Under Assumption \ref{basic_ass}, we have $- \frac{2\lambda\beta}{T}\log M_\lambda \to f_*$ as $\lambda \downarrow 0$.
\end{lemma}

\begin{proof}
Recall that $\alpha_\lambda = \frac{T}{2\lambda\beta}$. Since $f \geq f_*$, for every $\alpha_\lambda \geq 1$ we have
$$
M_\lambda
=
\int_{\R^d} e^{-\alpha_\lambda f(y)}\dd y
=
\int_{\R^d} e^{-(\alpha_\lambda - 1)f(y)}e^{-f(y)}\dd y
\leq
e^{-(\alpha_\lambda - 1)f_*}\int_{\R^d} e^{-f(y)}\dd y.
$$
Hence
$$
- \frac{1}{\alpha_\lambda}\log M_\lambda \geq f_* - \frac{1}{\alpha_\lambda}\bigg( \log \int_{\R^d} e^{-f(y)}\dd y + f_* \bigg),
$$
and the right-hand side tends to $f_*$.

Conversely, fix $\varepsilon > 0$. By continuity of $f$ at any minimizer point $x^*$, there exists $\rho > 0$ such that $f(y) \leq f_* + \varepsilon$ for every $y \in B_\rho(x^*)$.
Therefore $M_\lambda \geq \vert B_\rho \vert e^{-\alpha_\lambda(f_* + \varepsilon)}$, which yields
$$
- \frac{1}{\alpha_\lambda}\log M_\lambda \leq f_* + \varepsilon - \frac{1}{\alpha_\lambda}\log \vert B_\rho \vert.
$$
Letting $\lambda \downarrow 0$ and then $\varepsilon \downarrow 0$ gives the conclusion.
\end{proof}

The next lemma strengthens the usual weak concentration statement by providing an exponential estimate for $\pi_\lambda$.

\begin{lemma}\label{gibbs_concentration}
Under Assumptions \ref{basic_ass} and \ref{unique_min}, for every $r > 0$ there exist constants $C_r > 0$, $c_r > 0$, and $\lambda_r > 0$ such that
$$
\pi_\lambda\big( \Vert y - x^* \Vert \geq r \big) \leq C_r e^{-c_r/\lambda}
\qquad
\text{for every } \lambda \in (0,\lambda_r].
$$
In particular, $\pi_\lambda(y)\dd y \Rightarrow \delta_{x^*}$ (narrow convergence) as $\lambda \downarrow 0$.
\end{lemma}

\begin{proof}
Fix $r > 0$. Since $x^*$ is the unique minimizer and $f$ is continuous,
$$
m_r = \inf_{\Vert y - x^* \Vert \geq r} f(y) > f_*.
$$
Choose $\varepsilon_r \in \big( 0, m_r - f_* \big)$ and then $\rho_r \in (0,r)$ such that $f(y) \leq f_* + \varepsilon_r$ for every $\Vert y - x^* \Vert \leq \rho_r$.
As in the proof of Lemma \ref{partition_laplace}, $M_\lambda \geq \vert B_{\rho_r} \vert e^{-\alpha_\lambda(f_* + \varepsilon_r)}$.
For $\lambda$ small enough we have $\alpha_\lambda \geq 1$, and then
$$
\int_{\Vert y - x^* \Vert \geq r} e^{-\alpha_\lambda f(y)}\dd y
= \int_{\Vert y - x^* \Vert \geq r} e^{-(\alpha_\lambda - 1)f(y)}e^{-f(y)}\dd y 
\leq e^{-(\alpha_\lambda - 1)m_r}\int_{\R^d} e^{-f(y)}\dd y.
$$
Therefore
$$
\pi_\lambda\big( \Vert y - x^* \Vert \geq r \big)
\leq \frac{e^{-(\alpha_\lambda - 1)m_r}\int_{\R^d} e^{-f(y)}\dd y}{\vert B_{\rho_r} \vert e^{-\alpha_\lambda(f_* + \varepsilon_r)}} 
= \frac{e^{m_r}\int_{\R^d} e^{-f(y)}\dd y}{\vert B_{\rho_r} \vert} \exp \bigg( - \alpha_\lambda\big( m_r - f_* - \varepsilon_r \big) \bigg).
$$
Since $\displaystyle \alpha_\lambda = \frac{T}{2\lambda\beta}$ and $m_r - f_* - \varepsilon_r > 0$, this gives the desired estimate.
\end{proof}

To convert the concentration of $\pi_\lambda$ into uniform statements in $(t,x)$ on compact subsets of $[0,T) \times \R^d$, we use the following elementary testing lemma.

\begin{lemma}\label{uniform_test}
Let $\mu_n$ be a sequence of probability measures on $\R^d$ converging narrowly to $\delta_{x^*}$. Let $\mathscr F$ be an equibounded family of real-valued functions on $\R^d$ that is equicontinuous at $x^*$. Then
$$
\sup_{\varphi \in \mathscr F} \bigg\vert \int_{\R^d} \varphi(y)\mu_n(\dd y) - \varphi(x^*) \bigg\vert \to 0
\qquad
\text{as } n \to +\infty.
$$
\end{lemma}

\begin{proof}
Fix $\varepsilon > 0$. By equicontinuity at $x^*$, there exists $r > 0$ such that
$$
\sup_{\varphi \in \mathscr F} \sup_{\Vert y - x^* \Vert \leq r} \vert \varphi(y) - \varphi(x^*) \vert \leq \varepsilon.
$$
Let $M = \sup_{\varphi \in \mathscr F} \Vert \varphi \Vert_{L^\infty}$. Then
$$
\sup_{\varphi \in \mathscr F} \bigg\vert \int_{\R^d} \varphi(y)\mu_n(\dd y) - \varphi(x^*) \bigg\vert
\leq \sup_{\varphi \in \mathscr F} \int_{\R^d} \vert \varphi(y) - \varphi(x^*) \vert \mu_n(\dd y)
\leq \varepsilon + 2M\mu_n\big( \Vert y - x^* \Vert > r \big).
$$
Since $\mu_n$ converges narrowly to $\delta_{x^*}$, the second term tends to $0$. Letting $\varepsilon \to 0$ concludes the proof.
\end{proof}

The next proposition makes the global-selection mechanism fully explicit at the level of the conditional terminal law.

\begin{proposition}\label{conditional_mass_concentration}
Under Assumptions \ref{basic_ass} and \ref{unique_min}, for every compact subset $K \Subset [0,T) \times \R^d$ and every $r > 0$, there exist constants $C_{K,r} > 0$, $c_{K,r} > 0$, and $\lambda_{K,r} > 0$ such that
$$
\sup_{(t,x) \in K} \eta_{\lambda,t,x}\big( \Vert y - x^* \Vert \geq r \big) \leq C_{K,r} e^{-c_{K,r}/\lambda}
\qquad
\text{for every } \lambda \in (0,\lambda_{K,r}].
$$
\end{proposition}

\begin{proof}
Fix a compact set $K \Subset [0,T) \times \R^d$ and $r > 0$. There exists $\delta_K > 0$ such that $T-t \geq \delta_K$ for every $(t,x) \in K$, and the $x$-projection of $K$ is contained in some ball $B_R(0)$.
As above, let
$$
m_r = \inf_{\Vert y - x^* \Vert \geq r} f(y) > f_*.
$$
Choose $\varepsilon_r \in \big( 0, m_r - f_* \big)$ and then $\rho_r \in (0,r)$ such that $f(y) \leq f_* + \varepsilon_r$ whenever $\Vert y - x^* \Vert \leq \rho_r$.
Write
$$
D_{\lambda,t,x} = \int_{\R^d} G_{T-t}(x-y)e^{-\alpha_\lambda f(y)}\dd y.
$$
Since $(t,x,y)$ ranges over the compact set $\big\{ (t,x,y) \ \big\vert\ (t,x) \in K,\ \Vert y - x^* \Vert \leq \rho_r \big\}$, and $T-t \geq \delta_K$, the heat kernel admits a positive lower bound there:
$$
m_{K,r} = \inf \big\{ G_{T-t}(x-y) \ \big\vert\ (t,x) \in K,\ \Vert y - x^* \Vert \leq \rho_r \big\} > 0.
$$
Hence $D_{\lambda,t,x} \geq m_{K,r}\vert B_{\rho_r} \vert e^{-\alpha_\lambda(f_* + \varepsilon_r)}$ for all $(t,x) \in K$.

On the other hand, for $\lambda$ small enough so that $\alpha_\lambda \geq 1$,
\begin{align*}
\int_{\Vert y - x^* \Vert \geq r} G_{T-t}(x-y)e^{-\alpha_\lambda f(y)}\dd y
&\leq
(4\pi\beta\delta_K)^{-d/2}\int_{\Vert y - x^* \Vert \geq r} e^{-\alpha_\lambda f(y)}\dd y \\
&\leq
(4\pi\beta\delta_K)^{-d/2}e^{-(\alpha_\lambda - 1)m_r}\int_{\R^d} e^{-f(y)}\dd y.
\end{align*}
Dividing by the lower bound for $D_{\lambda,t,x}$ yields
$$
\eta_{\lambda,t,x}\big( \Vert y - x^* \Vert \geq r \big)
\leq
\frac{(4\pi\beta\delta_K)^{-d/2}e^{m_r}\int_{\R^d} e^{-f(y)}\dd y}{m_{K,r}\vert B_{\rho_r} \vert}
\exp \left( - \alpha_\lambda \big( m_r - f_* - \varepsilon_r \big) \right),
$$
uniformly in $(t,x) \in K$. Since $\displaystyle \alpha_\lambda = \frac{T}{2\lambda\beta}$, the conclusion follows.
\end{proof}

\begin{theorem}\label{conditional_low_temp}
Under Assumptions \ref{basic_ass} and \ref{unique_min}, the probability measures $\eta_{\lambda,t,x}(y)\dd y$ converge to $\delta_{x^*}$ as $\lambda \downarrow 0$ uniformly for $(t,x)$ in compact subsets of $[0,T) \times \R^d$, i.e.,
$$
\sup_{(t,x) \in K} \bigg\vert \int_{\R^d} \varphi(y)\eta_{\lambda,t,x}(y)\dd y - \varphi(x^*) \bigg\vert \to 0
\qquad
\text{as } \lambda \downarrow 0,
$$
for every compact subset $K \Subset [0,T) \times \R^d$ and every bounded continuous function $\varphi : \R^d \to \R$.
\end{theorem}

\begin{proof}
Fix $K \Subset [0,T) \times \R^d$, a bounded continuous function $\varphi$, and $\varepsilon > 0$. By continuity of $\varphi$ at $x^*$, there exists $r > 0$ such that $\vert \varphi(y) - \varphi(x^*) \vert \leq \varepsilon$ whenever $\Vert y - x^* \Vert \leq r$.
Then
\begin{align*}
\sup_{(t,x) \in K} \bigg\vert \int_{\R^d} \varphi(y)\eta_{\lambda,t,x}(y)\dd y - \varphi(x^*) \bigg\vert
&\leq
\sup_{(t,x) \in K}\int_{\R^d} \vert \varphi(y) - \varphi(x^*) \vert \eta_{\lambda,t,x}(y)\dd y \\
&\leq
\varepsilon + 2\Vert \varphi \Vert_{L^\infty} \sup_{(t,x) \in K}\eta_{\lambda,t,x}\big( \Vert y - x^* \Vert \geq r \big).
\end{align*}
The second term tends to $0$ by Proposition \ref{conditional_mass_concentration}. Letting $\varepsilon \downarrow 0$ concludes the proof.
\end{proof}

Theorem \ref{conditional_low_temp} is the global-optimization content of the finite-horizon criterion. Because Assumption \ref{unique_min} allows any number of nonglobal local minima, the theorem says that these local traps disappear in the low-temperature limit at the level of the exact continuous-time criterion: conditionally on the current state $(t,x)$, the terminal law still collapses onto the global minimizer.

We now pass from the conditional law to the drift field and the value function themselves.

\begin{theorem}\label{small_lambda}
Under Assumptions \ref{basic_ass} and \ref{unique_min}, we have
$$
V_\lambda(t,x) \to f_*,
\qquad
a_\lambda(t,x) \to x^*,
\qquad
\mathbf u_\lambda^*(t,x) \to - \frac{x - x^*}{T-t},
$$
as $\lambda \downarrow 0$, uniformly on any compact subset $K \Subset [0,T) \times \R^d$.
\end{theorem}
\begin{proof}
Since $K \Subset [0,T) \times \R^d$, there exists $\delta_K > 0$ such that $T-t \geq \delta_K$ for all $(t,x) \in K$, and the $x$-projection of $K$ is compact. For $(t,x) \in K$, define
$g_{t,x}(y) = G_{T-t}(x-y)$ and $b_{t,x}^i(y) = y_i G_{T-t}(x-y)$, for $i = 1,\dots,d$.
Since $K$ is compact and $T-t \geq \delta_K$, the families
$\mathscr G_K = \big\{ g_{t,x} \ \big\vert\ (t,x) \in K \big\}$
and
$\mathscr B_K^i = \big\{ b_{t,x}^i \ \big\vert\ (t,x) \in K \big\}$
are equibounded and equicontinuous on $\R^d$. By Lemmas \ref{gibbs_concentration} and \ref{uniform_test},
\begin{equation}\label{uniform_h}
\sup_{(t,x) \in K} \vert h_\lambda(t,x) - G_{T-t}(x-x^*) \vert \to 0,
\end{equation}
and, for each $i = 1,\dots,d$,
\begin{equation}\label{uniform_num}
\sup_{(t,x) \in K} \bigg\vert \int_{\R^d} y_i G_{T-t}(x-y)\pi_\lambda(y)\dd y - x_i^* G_{T-t}(x-x^*) \bigg\vert \to 0.
\end{equation}
Since $G_{T-t}(x-x^*)$ is strictly positive on $K$, \eqref{uniform_h} and \eqref{uniform_num} imply
\begin{equation}\label{eq:alambda-limit}
\sup_{(t,x) \in K} \Vert a_\lambda(t,x) - x^* \Vert \to 0.
\end{equation}
The barycentric representation \eqref{def_a_lambda} then gives
$$
\sup_{(t,x) \in K} \bigg\Vert \mathbf u_\lambda^*(t,x) + \frac{x - x^*}{T-t} \bigg\Vert \to 0.
$$
Finally,
$\displaystyle V_\lambda(t,x) = - \frac{2\lambda\beta}{T}\log h_\lambda(t,x) - \frac{2\lambda\beta}{T}\log M_\lambda$.
By Lemma \ref{partition_laplace}, $\displaystyle - \frac{2\lambda\beta}{T}\log M_\lambda \to f_*$.
By \eqref{uniform_h}, $h_\lambda(t,x)$ converges uniformly on $K$ to the strictly positive function $G_{T-t}(x-x^*)$, and therefore $\log h_\lambda$ is uniformly bounded on $K$ for all sufficiently small $\lambda$. Since $\displaystyle \frac{2\lambda\beta}{T} \to 0$, we conclude that
$$
\sup_{(t,x) \in K} \bigg\vert \frac{2\lambda\beta}{T}\log h_\lambda(t,x) \bigg\vert \to 0.
$$
Hence $V_\lambda \to f_*$ uniformly on $K$.
\end{proof}

Theorem \ref{small_lambda} should be read as the field-level counterpart of Theorem \ref{conditional_low_temp}. Away from the terminal singularity, the exact finite-horizon drift converges uniformly on compact sets to the affine field $\displaystyle x \mapsto - \frac{x - x^*}{T-t}$, which points toward the global minimizer. This is the stochastic feedback analogue of Corollary \ref{warmup_small_lambda}.

\begin{remark}[Non-commutativity of the two limits]\label{non_commute}
The two asymptotic regimes of this section do not commute. For fixed $\lambda > 0$, Proposition \ref{terminal_limit} gives $\mathbf u_\lambda^*(t,x) \to - \frac{T}{\lambda}\nabla f(x)$ as $t \uparrow T$. For fixed $t < T$, Theorem \ref{small_lambda} gives $\displaystyle \mathbf u_\lambda^*(t,x) \to - \frac{x - x^*}{T-t}$ as $\lambda \downarrow 0$. At a point $x \neq x^*$ where $\nabla f(x) \neq 0$, these two limits are generically different and point in different directions. At the global minimizer $x = x^*$, both limits vanish. In other words, the global mechanism (low $\lambda$) and the local mechanism (near terminal time) can compete, and the behavior of the drift when $\lambda \downarrow 0$ and $t \uparrow T$ simultaneously depends on the relative rates of the two limits.
\end{remark}

We now translate these pointwise statements into assertions about the optimally controlled process itself.

\begin{corollary}\label{terminal_process_concentration}
Let $X^\lambda$ be the optimally controlled process, started from an initial law $\rho_0$ that does not depend on $\lambda$ and is supported in a compact set $K_0 \subset \R^d$. Under Assumptions \ref{basic_ass} and \ref{unique_min}, we have $\Law(X_T^\lambda) \Rightarrow \delta_{x^*}$ as $\lambda \downarrow 0$, i.e., $\E \big[ \varphi(X_T^\lambda) \big] \to \varphi(x^*)$ for any bounded continuous function $\varphi$.
\end{corollary}

\begin{proof}
By conditioning on $X_0^\lambda = x$ and using Lemma \ref{kernel_lemma},
$\Law(X_T^\lambda)(\dd y) = \int_{K_0} \eta_{\lambda,0,x}(y)\rho_0(\dd x)$.
Hence, for every bounded continuous $\varphi$,
$$
\bigg\vert \E \big[ \varphi(X_T^\lambda) \big] - \varphi(x^*) \bigg\vert
\leq \sup_{x \in K_0} \bigg\vert \int_{\R^d} \varphi(y)\eta_{\lambda,0,x}(y)\dd y - \varphi(x^*) \bigg\vert.
$$
The right-hand side tends to $0$ by Theorem \ref{conditional_low_temp}, applied to the compact set $\{0\} \times K_0$.
\end{proof}

\begin{corollary}\label{objective_exact_init}
Let $X^\lambda$ be the optimal process started from the exact density $h_\lambda(0,\cdot)$. Under Assumptions \ref{basic_ass} and \ref{unique_min}, we have $\Law(X_T^\lambda) = \pi_\lambda(y)\dd y$ and
$$
\E \big[ f(X_T^\lambda) \big] = \int_{\R^d} f(y)\pi_\lambda(y)\dd y \to f_*
\qquad
\text{as } \lambda \downarrow 0.
$$
\end{corollary}

\begin{proof}
The identity for the terminal law is given by Corollary \ref{exact_init}. Since $f \geq f_*$, it remains to prove the upper bound. Fix $\varepsilon > 0$. By continuity of $f$ at $x^*$, there exists $r > 0$ such that
$f(y) \leq f_* + \varepsilon$ whenever $\Vert y - x^* \Vert \leq r$.
Then
\begin{align*}
\int_{\R^d} f(y)\pi_\lambda(y)\dd y
&\leq (f_* + \varepsilon)\pi_\lambda\big( \Vert y - x^* \Vert \leq r \big) + \int_{\Vert y - x^* \Vert > r} f(y)\pi_\lambda(y)\dd y \\
&\leq f_* + \varepsilon + \frac{1}{M_\lambda}\int_{\Vert y - x^* \Vert > r} \vert f(y) \vert e^{-\alpha_\lambda f(y)}\dd y.
\end{align*}
By coercivity, $\vert f \vert e^{-f} \in L^1(\R^d)$. Setting $m_r = \inf_{\Vert y - x^* \Vert > r} f(y)$, we have $m_r > f_*$. Arguing as in the proof of Lemma \ref{gibbs_concentration}, there exists $\rho_r > 0$ such that $M_\lambda \geq \vert B_{\rho_r} \vert e^{-\alpha_\lambda(f_* + \varepsilon)}$, and therefore
$$
\frac{1}{M_\lambda}\int_{\Vert y - x^* \Vert > r} \vert f(y) \vert e^{-\alpha_\lambda f(y)}\dd y
\leq \frac{e^{-(\alpha_\lambda - 1)m_r}}{\vert B_{\rho_r} \vert e^{-\alpha_\lambda(f_* + \varepsilon)}}\int_{\R^d} \vert f(y) \vert e^{-f(y)}\dd y
\to 0.
$$
Thus
$$
\limsup_{\lambda \downarrow 0} \int_{\R^d} f(y)\pi_\lambda(y)\dd y \leq f_* + \varepsilon.
$$
Since $\varepsilon > 0$ is arbitrary and the reverse inequality is trivial, the result follows.
\end{proof}

Corollaries \ref{terminal_process_concentration} and \ref{objective_exact_init} give the process-level meaning of the low-temperature limit. The exact continuous-time optimizer selected by \eqref{criterion_intro} drives the terminal law toward the global minimizer as $\lambda \downarrow 0$; for the exact self-consistent initialization, even the expected terminal objective converges to the global minimum value.

\subsection{Nondegenerate case and Laplace asymptotics}

The previous results use only low-temperature concentration. When the minimizer is nondegenerate, one can sharpen them by a leading-order Laplace expansion.

\begin{assumption}\label{nondegenerate_ass}
In addition to Assumptions \ref{basic_ass} and \ref{unique_min}, $f \in C^4$ in a neighborhood of $x^*$ and $H_* = D^2 f(x^*)$ is positive definite.
\end{assumption}

\begin{proposition}\label{laplace_prop}
Set $\displaystyle C_* = \frac{(2\pi)^{d/2}}{\sqrt{\det H_*}}$ and recall 
 $\displaystyle\alpha_\lambda = \frac{T}{2\lambda\beta}$.
Under Assumption \ref{nondegenerate_ass}, we have
\begin{align}
\displaystyle \int_{\R^d} G_{T-t}(x-y)e^{-\alpha_\lambda f(y)}\dd y
&=
 \displaystyle e^{-\alpha_\lambda f_*}\alpha_\lambda^{-d/2} C_* G_{T-t}(x-x^*)\big( 1 + \mathrm{O}(\lambda) \big), \label{laplace_integral} \\
\displaystyle h_\lambda(t,x)
&=
 \displaystyle G_{T-t}(x-x^*)\big( 1 + \mathrm{O}(\lambda) \big), \label{laplace_h} \\[2mm]
\displaystyle a_\lambda(t,x)
&=
 \displaystyle x^* + \mathrm{O}(\lambda), \label{laplace_a} \\
\displaystyle \mathbf u_\lambda^*(t,x)
&=
\displaystyle - \frac{x - x^*}{T-t} + \mathrm{O}(\lambda), \label{laplace_u} \\
\displaystyle V_\lambda(t,x)
&=
\displaystyle f_* + \frac{2\lambda\beta}{T}\bigg[ \frac{d}{2}\log \alpha_\lambda - \log \Big( C_* G_{T-t}(x-x^*) \Big) \bigg] + \mathrm{O}(\lambda^2), \label{laplace_V}
\end{align}
as $\lambda \downarrow 0$, uniformly on any compact subset $K \Subset [0,T) \times \R^d$.
\end{proposition}

\begin{proof}
Set $g_{t,x}(y) = G_{T-t}(x-y)$.
Since $K \Subset [0,T) \times \R^d$, the family $\{ g_{t,x} \}_{(t,x) \in K}$ is uniformly bounded in $C^2(U)$ on some fixed neighborhood $U$ of $x^*$. The multidimensional Laplace method with parameter-dependent amplitudes (see, for example, \cite[Chapter 4]{deBruijn81} and \cite[Chapter IX]{Wong01}) therefore yields, uniformly on $K$,
$$
\int_{\R^d} g_{t,x}(y)e^{-\alpha_\lambda f(y)}\dd y
=
 e^{-\alpha_\lambda f_*}\alpha_\lambda^{-d/2} C_* \Big( g_{t,x}(x^*) + \mathrm{O}(\alpha_\lambda^{-1}) \Big).
$$
Since $g_{t,x}(x^*) = G_{T-t}(x-x^*)$ and $\alpha_\lambda^{-1} = \frac{2\lambda\beta}{T}$, this proves \eqref{laplace_integral}.
Applying the same expansion with amplitude identically equal to $1$ gives
$M_\lambda = e^{-\alpha_\lambda f_*}\alpha_\lambda^{-d/2} C_* \big( 1 + \mathrm{O}(\lambda) \big)$.
Dividing \eqref{laplace_integral} by this expression yields \eqref{laplace_h}. Applying the same argument once more with amplitudes $g_{t,x}^{(i)}(y) = y_i\, G_{T-t}(x-y)$ gives
$$
\int_{\R^d} y_i\, G_{T-t}(x-y)e^{-\alpha_\lambda f(y)}\dd y
= e^{-\alpha_\lambda f_*}\alpha_\lambda^{-d/2} C_* \Big( x_i^* G_{T-t}(x-x^*) + \mathrm{O}(\lambda) \Big),
$$
uniformly on $K$. Dividing by the denominator in \eqref{laplace_integral} gives \eqref{laplace_a}. The barycentric formula from Proposition \ref{three_forms} then yields \eqref{laplace_u}.
Finally, using \eqref{laplace_integral},
\begin{align*}
V_\lambda(t,x)
&=
- \frac{1}{\alpha_\lambda}\log \bigg( e^{-\alpha_\lambda f_*}\alpha_\lambda^{-d/2} C_* G_{T-t}(x-x^*)\big( 1 + \mathrm{O}(\lambda) \big) \bigg) \\
&=
 f_* + \frac{1}{\alpha_\lambda}\bigg[ \frac{d}{2}\log \alpha_\lambda - \log \Big( C_* G_{T-t}(x-x^*) \Big) \bigg] - \frac{1}{\alpha_\lambda}\log \big( 1 + \mathrm{O}(\lambda) \big).
\end{align*}
Since $\displaystyle \alpha_\lambda^{-1} = \frac{2\lambda\beta}{T}$ and $\log \big( 1 + \mathrm{O}(\lambda) \big) = \mathrm{O}(\lambda)$, the last term is $\mathrm{O}(\lambda^2)$, which proves \eqref{laplace_V}.
\end{proof}

Proposition \ref{laplace_prop} sharpens the low-temperature theory in two ways. First, the drift correction is of order $\lambda$ on compact subsets away from terminal time. Second, the value function has an explicit first correction of size $\lambda \log \big( \frac{1}{\lambda} \big)$ coming from the factor $\alpha_\lambda^{-d/2}$. This is substantially more informative than a mere convergence statement.

\begin{corollary}[Asymptotic covariance of the conditional terminal law]\label{laplace_covariance}
Under Assumption \ref{nondegenerate_ass}, we have
$$
\Covtx
= \frac{2\lambda\beta}{T}H_*^{-1} + \mathrm{O}(\lambda^2),
$$
as $\lambda \downarrow 0$, uniformly for $(t,x) \in K$, for any compact subset $K \Subset [0,T) \times \R^d$, where 
\begin{equation}\label{eq:covariance}
\Covtx = \int_{\R^d} (y - a_\lambda(t,x))(y - a_\lambda(t,x))^\top \eta_{\lambda,t,x}(y)\dd y.
\end{equation}
\end{corollary}

\begin{proof}
The penalized energy \eqref{penalized_energy} satisfies $\displaystyle D_y^2 \mathcal E_{\lambda,t,x}(x^*) = H_* + \frac{\lambda}{T(T-t)}I_d$. Since $\eta_{\lambda,t,x}$ is a Gibbs law at inverse temperature $\alpha_\lambda$ for this energy, the standard Laplace argument for second moments (see, for example, \cite[Chapter~4]{deBruijn81}) gives
$$
\Covtx
=
\frac{1}{\alpha_\lambda}\bigg( H_* + \frac{\lambda}{T(T-t)}I_d \bigg)^{-1} + \mathrm{O}(\alpha_\lambda^{-2}).
$$
Since $\displaystyle \alpha_\lambda^{-1} = \frac{2\lambda\beta}{T}$, expanding the inverse matrix to first order gives the result uniformly on $K$, because $T-t \geq \delta_K > 0$ on $K$.
\end{proof}

Thus, in the nondegenerate regime, the conditional terminal law is approximately Gaussian with mean $x^* + \mathrm{O}(\lambda)$ and covariance $(2\lambda\beta/T)H_*^{-1}$. The directions in which $f$ is flatter at $x^*$ (small eigenvalues of $H_*$) produce larger spread in the terminal law, as expected.

\medskip

\paragraph{\bf Low-temperature behavior of the OSL coefficient.}
As a direct consequence of Corollary \ref{laplace_covariance}, we obtain the low-temperature behavior of the trace coefficient defined in \eqref{eq:Ceta}. For $t$ in a compact subset of $[0,T)$ and $(x_1,x_2)$ in a compact subset of $\R^d\times\R^d$, the segment $x(\theta)=x_1+\theta(x_2-x_1)$ stays bounded, and Corollary \ref{laplace_covariance} gives
\begin{equation}\label{Ceta_laplace}
C_{\eta_\lambda}(t,x_1,x_2)
=\frac{2\lambda\beta}{T}\operatorname{tr}(H_*^{-1})+\mathrm{O}(\lambda^2),
\end{equation}
uniformly in that set. Equivalently, since $T-t$ is bounded away from zero on such compact subsets,
$$
\int_0^1\big(d-(T-t)\Delta\Phi_\lambda(t,x(\theta))\big)\dd\theta
=
\frac{\lambda}{T(T-t)}\operatorname{tr}(H_*^{-1})+\mathrm{O}(\lambda^2).
$$
Thus the factorized trace OSL coefficient is of order $\lambda$ on compact subsets away from the terminal time. Exponential estimates concern the tails away from neighborhoods of $x^*$, whereas the full covariance is generically of order $\lambda$ in the nondegenerate Laplace regime.

\section{Gradient-free discretization}\label{numerics_sec}
The continuous finite-horizon theory produces a natural gradient-free drift formula \eqref{def_a_lambda}, which explains the underlying exploration-exploitation mechanism. The purpose of this section is not to develop a full complexity theory. It is simply to show how the exact barycentric representation \eqref{def_a_lambda} suggests a practical gradient-free discretization and to indicate how the parameters $T$, $\lambda$, $N$, and $h$ should be interpreted.

Suppose that at time $t$ the current position is $x$. To approximate the barycenter $a_\lambda(t,x)$ from \eqref{def_a_lambda}, one may draw $Y^1,\dots,Y^N$ independently from the Gaussian law $\mathcal N(x,2\beta(T-t)I_d)$ and use the importance-sampling estimator
\begin{equation}\label{mc_barycenter}
a_{\lambda,N}(t,x) = \frac{\sum_{j=1}^N Y^j \exp \big( - \alpha_\lambda f(Y^j) \big)}{\sum_{j=1}^N \exp \big( - \alpha_\lambda f(Y^j) \big)}.
\end{equation}
This yields the approximate drift
\begin{equation}\label{approx_drift}
\mathbf u_{\lambda,N}(t,x) = - \frac{x - a_{\lambda,N}(t,x)}{T-t}.
\end{equation}
A simple Euler-Maruyama discretization (see \cite{KloedenPlaten92}) on the grid $t_k = kh$, stopped at $t_K = T - \delta$ with $\delta > 0$, reads
\begin{equation}\label{em_scheme}
X_{k+1} = X_k + h\mathbf u_{\lambda,N}(t_k,X_k) + \sqrt{2\beta h}\xi_k,
\qquad
\xi_k \sim \mathcal N(0,I_d),
\qquad
k = 0,\dots,K-1.
\end{equation}
The parameters have the following interpretation:
\begin{enumerate}[label=\roman*)]
\item $T$ fixes the horizon and therefore the search radius $\sqrt{2\beta(T-t)}$ at time $t$.
\item $\lambda$ fixes the Gibbs concentration through $\alpha_\lambda = \frac{T}{2\lambda\beta}$: small $\lambda$ means low temperature and stronger concentration near minimizers.
\item $N$ controls the Monte Carlo approximation of the barycenter $a_\lambda(t,x)$ in \eqref{mc_barycenter}.
\item $h$ is the time discretization step.
\end{enumerate}
A satisfactory non-asymptotic theorem for \eqref{em_scheme} should combine the time discretization error, the Monte Carlo variance, the near-terminal stiffness coming from the factor $(T-t)^{-1}$, the low-temperature covariance scaling in \eqref{Ceta_laplace} and a fair comparison at equal computational budget with other global methods. A detailed description of the computational algorithm along these lines is left to a follow-up work. 

Figure \ref{fig:shrinking} visualizes that mechanism on the two-dimensional Ackley landscape. The colormap represents the conditional terminal density $\eta_{\lambda,t,x}$, the cyan dot is the current position $x$, the red star is the minimizer, and the white arrow is the drift $\mathbf u_\lambda^*(t,x)$. As $t$ increases, the conditional terminal law shrinks and the drift passes from a genuinely global barycentric direction to a local direction. We stress that this figure is only qualitative: the Ackley function is used as a familiar visual benchmark, but it does not exactly satisfy Assumption \ref{basic_ass}.

\begin{figure}[h]
\centering
\includegraphics[width=\textwidth]{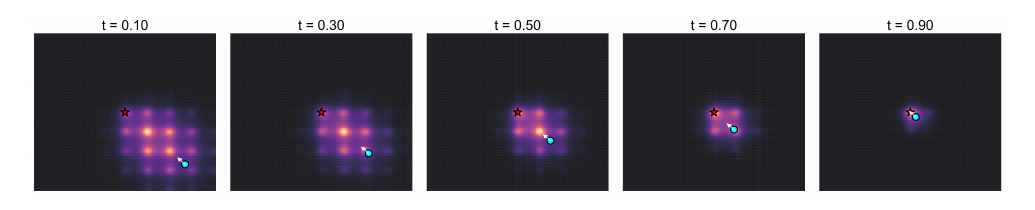}
\caption{Visualization of the conditional terminal density $\eta_{\lambda,t,x}$ and of the barycentric drift $\mathbf u_\lambda^*(t,x)$ at several times. Early times correspond to broad terminal averaging, while late times recover a local direction.}
\label{fig:shrinking}
\end{figure}

\section{Conclusion and further comments}\label{further_comments_sec}

This paper studies the finite-horizon criterion \eqref{criterion_intro} as an optimization problem over drift fields. Starting from the deterministic warm-up, we solved the corresponding stochastic control problem on $\R^d$, identified the exact Markov feedback, and reorganized it around the conditional terminal law of the optimal process. This law is the Gibbs measure of the penalized energy \eqref{penalized_energy}; it yields the potential, averaged-gradient, and barycentric representations \eqref{potential_form}, \eqref{gradient_and_barycenter} and \eqref{def_a_lambda}. 

We then analyzed two asymptotic regimes that are directly relevant for optimization. As $t \uparrow T$, the drift reconnects with a scaled gradient descent. As $\lambda \downarrow 0$, the conditional terminal law, the optimal process, and the drift select the unique global minimizer, even when $f$ has nonglobal local minima; in the nondegenerate case we also obtained leading-order Laplace asymptotics for the drift, the value function, and the covariance of the conditional terminal law. The two regimes are generically non-commutative (Remark \ref{non_commute}), which reflects the tension between global exploration and local exploitation.

Finally, we recorded a simple gradient-free discretization suggested by the exact barycentric formula.

The Hopf-Cole closed form and its relation with linearly solvable control, reciprocal diffusions, mean-field games, and recent Hamilton-Jacobi or Wasserstein-proximal viewpoints are classical. The present paper focuses on a narrower aspect which, we believe, enables a sharper message: it isolates the optimization content of that classical formula in a coherent framework on $\R^d$; it places the conditional terminal law at the center of the discussion; it makes the endogenous search radius $\sqrt{2\beta(T-t)}$ and temperature $\displaystyle \frac{2\lambda\beta}{T}$ completely explicit; and it states the low-temperature asymptotics in a form directly relevant for global optimization. Thus, the novelty is not a new Hopf-Cole transform, but the systematic extraction of the optimizer selected by this transform and the precise global-selection interpretation of its conditional terminal law. In particular, the paper highlights that the exact continuous-time criterion selects the global minimizer asymptotically under uniqueness of that minimizer, without any assumption excluding non-global local minima.

We conclude with a few further comments and perspectives.

\medskip
\paragraph{\bf Bounded domains and reflection.}
Throughout the paper we have deliberately worked on the whole space $\R^d$. If one wishes instead to impose a hard state constraint on a smooth bounded domain $\Omega \subset \R^d$, the coherent way to do so is to replace the free diffusion on $\R^d$ by a reflected diffusion
$$
X_t = x_0 + \int_0^t u_s \dd s + \sqrt{2\beta}B_t - \int_0^t n(X_s)\dd K_s,
$$
where $n$ is the outward unit normal and $K_t$ is the boundary local time. At the PDE level, the corresponding Fokker-Planck equation carries a no-flux condition and the backward equation for $h_\lambda$ becomes the backward heat equation with Neumann boundary condition. In that reflected setting, a counterpart of the present analysis should be possible with the Neumann heat semigroup. The local terminal cancellation discussed in Remark \ref{terminal_semiconcavity_rem} should then be understood only away from boundary effects. Near the boundary, reflected heat-kernel contributions and the geometry of the boundary may alter the Riccati/caustic picture, and the covariance factorization used in the whole-space proof must be replaced by its reflected analogue. By contrast, declaring $f = +\infty$ outside $\Omega$ while still evolving a free diffusion on $\R^d$ mixes two different models and is not a satisfactory way to encode state constraints.

\medskip
\paragraph{\bf Several global minimizers.}
If $f$ has several global minimizers, the Gibbs measures $\pi_\lambda(y)\dd y$ still concentrate on the set $\operatorname*{argmin} f$, but the limiting distribution need not be a single Dirac mass. Accordingly, the barycenter $a_\lambda(t,x)$ may depend in a nontrivial way on the Gaussian factor $G_{T-t}(x-\cdot)$ and on the geometry of the minimizer set. In that regime one should not expect the single affine limit $\displaystyle - \frac{x - x^*}{T-t}$ without additional symmetry-breaking assumptions.

\medskip
\paragraph{\bf Unique global minimizer and several local minima.}
This is the regime covered by Section \ref{asymp_sec}. The only structural assumption there is the uniqueness of the global minimizer; any number of nonglobal local minima or saddle points is allowed. The exponential concentration estimate of Proposition \ref{conditional_mass_concentration} shows that, at the level of the exact continuous-time criterion, those local minima do not affect the low-temperature selection mechanism: the terminal law still concentrates near $x^*$. What the paper does \emph{not} prove is a finite-$\lambda$ or complexity guarantee for a practical algorithm. For fixed small $\lambda$ and for the Monte Carlo discretization of Section \ref{numerics_sec}, metastability, sampling error, and near-terminal stiffness may still matter. Thus the result is a global-selection theorem for the exact drift in the asymptotic regime $\lambda \downarrow 0$, not yet a full computational-complexity statement.

\medskip
\paragraph{\bf Other running costs and control constraints.}
The quadratic running cost is special because it leads to the Hopf-Cole linearization of the Hamilton-Jacobi equation \eqref{hjb} via the heat equation \eqref{h_lambda_heat}, and therefore to the explicit formulas of Theorem \ref{main_value}, Proposition \ref{penalized_gibbs}, and Proposition \ref{three_forms}; see also \cite{Kappen05,Todorov06,Todorov09}. Replacing it by an $L^1$ running cost, or imposing a hard bound $\Vert u_t \Vert \leq 1$, leads to a completely different optimal control problem. From the Hamiltonian point of view, this is where a Pontryagin maximum principle analysis should become highly relevant. In the $L^1$ case the minimizing Hamiltonian is no longer quadratic and one expects nonsmooth Hamilton-Jacobi equations and possibly sparse drifts. Under the hard bound $\Vert u_t \Vert \leq 1$, the minimizing controls should saturate the constraint on the region where the adjoint gradient is nonzero, leading to bang-bang type feedbacks. These variants are mathematically natural from the optimization viewpoint, but they fall outside the linearly solvable framework of the present paper.


\begin{thebibliography}{99}

\bibitem{AndrieuChopinFincatoGerber24}
C. Andrieu, N. Chopin, E. Fincato, and M. Gerber,
\emph{Gradient-free optimization via integration},
arXiv:2408.00888, 2024.


\bibitem{AgrachevSachkov04}
A. A. Agrachev and Y. L. Sachkov,
\emph{Control Theory from the Geometric Viewpoint},
Encyclopaedia of Mathematical Sciences \textbf{87}, Springer, 2004.

\bibitem{BardiCapuzzoDolcetta97}
M. Bardi and I. Capuzzo-Dolcetta,
\emph{Optimal Control and Viscosity Solutions of Hamilton-Jacobi-Bellman Equations},
Birkh\"auser, Boston, 1997.

\bibitem{BonnardCaillauTrelat07}
B. Bonnard, J.-B. Caillau, and E. Tr\'elat,
Second order optimality conditions in the smooth case and applications in optimal control,
\emph{ESAIM Control Optim. Calc. Var.} \textbf{13} (2007), 207-236.


\bibitem{BensoussanFrehseYam13}
A. Bensoussan, J. Frehse, and P. Yam,
\emph{Mean Field Games and Mean Field Type Control Theory},
Springer, 2013.


\bibitem{CannarsaSinestrari04}
P. Cannarsa and C. Sinestrari,
\emph{Semiconcave Functions, Hamilton-Jacobi Equations, and Optimal Control},
Progress in Nonlinear Differential Equations and Their Applications \textbf{58}, Birkh\"auser, Boston, 2004.


\bibitem{CarrilloChoiTotzeckTse18}
J. A. Carrillo, Y.-P. Choi, C. Totzeck, and O. Tse,
An analytical framework for consensus-based global optimization method,
\emph{Math. Models Methods Appl. Sci.} \textbf{28} (2018), 1037-1066.

\bibitem{DaiJiaoKangLuYang23}
Y. Dai, Y. Jiao, L. Kang, X. Lu, and J. Z. Yang,
Global optimization via Schr\"odinger-F\"ollmer diffusion,
\emph{SIAM J. Control Optim.} \textbf{61} (2023), 2953-2980.

\bibitem{DaiPra91}
P. Dai Pra,
A stochastic control approach to reciprocal diffusion processes,
\emph{Appl. Math. Optim.} \textbf{23} (1991), 313-329.

\bibitem{deBruijn81}
N. G. de Bruijn,
\emph{Asymptotic Methods in Analysis},
third edition, Dover, 1981.

\bibitem{DemboZeitouni}
A. Dembo and O. Zeitouni,
\emph{Large Deviations Techniques and Applications},
Springer, 2010.

\bibitem{FlemingSoner06}
W. H. Fleming and H. M. Soner,
\emph{Controlled Markov Processes and Viscosity Solutions},
second edition, Springer, 2006.

\bibitem{Follmer85}
H. F\"ollmer,
An entropy approach to the time reversal of diffusion processes,
in \emph{Stochastic Differential Systems, Filtering and Control},
Lecture Notes in Control and Information Sciences \textbf{69},
Springer, 1985, pp.~156-163.

\bibitem{FornasierKlockRiedl24}
M. Fornasier, T. Klock, and K. Riedl,
Consensus-based optimization methods converge globally,
\emph{SIAM J. Optim.} \textbf{34} (2024), 2973-3004.

\bibitem{HeatonFungOsher24}
H. Heaton, S. Wu Fung, and S. Osher,
Global solutions to nonconvex problems by evolution of Hamilton-Jacobi PDEs,
\emph{Commun. Appl. Math. Comput.} \textbf{6} (2024), 790-810.

\bibitem{HuangJiaoKangLiaoLiuLiu21}
J. Huang, Y. Jiao, L. Kang, X. Liao, J. Liu, and Y. Liu,
\emph{Schr\"odinger-F\"ollmer sampler: sampling without ergodicity},
arXiv:2106.10880, 2021.

\bibitem{HuangKaliseKouhkouh25}
Y. Huang, D. Kalise, and H. Kouhkouh,
Non-convex global optimization as an optimal stabilization problem: dynamical properties,
\emph{arXiv preprint} arXiv:2511.10815, 2025.

\bibitem{Hwang80}
C.-R. Hwang,
Laplace's method revisited: weak convergence of probability measures,
\emph{Ann. Probab.} \textbf{8} (1980), 1177-1182.

\bibitem{IwakiriWangItoTakeda22}
H. Iwakiri, Y. Wang, S. Ito, and A. Takeda,
Single loop Gaussian homotopy method for non-convex optimization,
in \emph{Advances in Neural Information Processing Systems} \textbf{35},
2022, pp.~7065-7076.

\bibitem{Kappen05}
H. J. Kappen,
Path integrals and symmetry breaking for optimal control theory,
\emph{J. Stat. Mech. Theory Exp.} 2005, P11011.

\bibitem{KloedenPlaten92}
P. E. Kloeden and E. Platen,
\emph{Numerical Solution of Stochastic Differential Equations},
Springer, 1992.

\bibitem{kruzhkov1964cauchy}
S. N. Kruzhkov,
The Cauchy problem in the large for nonlinear equations and for certain quasilinear systems of first order with several variables,
\emph{Soviet Math. Dokl.} \textbf{5} (1964), 493-496.
 
\bibitem{LasryLions07}
J.-M. Lasry and P.-L. Lions,
Mean field games,
\emph{Japan. J. Math.} \textbf{2} (2007), 229-260.

\bibitem{LiLiuOsher23}
W. Li, S. Liu, and S. Osher,
A kernel formula for regularized Wasserstein proximal operators,
\emph{Res. Math. Sci.} \textbf{10} (2023), Article 43.

\bibitem{LiYong95}
X. Li and J. Yong,
\emph{Optimal Control Theory for Infinite Dimensional Systems},
Birkh\"auser, Boston, 1995.

\ifx
\bibitem{lin2001}
C.-T. Lin and E. Tadmor,
$L^1$-stability and error estimates for approximate Hamilton-Jacobi solutions,
\emph{Numer. Math.} \textbf{87} (2001), 701-735.
 \fi
 
\bibitem{lions1982generalized}
P.-L. Lions,
\emph{Generalized Solutions of Hamilton-Jacobi Equations},
Research Notes in Mathematics \textbf{69}, Pitman, 1982.

\bibitem{OsherHeatonFung23}
S. Osher, H. Heaton, and S. Wu Fung,
A Hamilton-Jacobi-based proximal operator,
\emph{Proc. Natl. Acad. Sci. USA} \textbf{120} (2023), e2220469120.

\bibitem{Pavliotis14}
G. A. Pavliotis,
\emph{Stochastic Processes and Applications: Diffusion Processes, the Fokker-Planck and Langevin Equations},
Springer, 2014.

\bibitem{PontryaginBook}
L. S. Pontryagin, V. G. Boltyanskii, R. V. Gamkrelidze, and E. F. Mishchenko,
\emph{The Mathematical Theory of Optimal Processes},
Interscience, New York, 1962.

\bibitem{Risken96}
H. Risken,
\emph{The Fokker-Planck Equation: Methods of Solution and Applications},
second edition, Springer, 1996.

\bibitem{Todorov06}
E. Todorov,
Linearly-solvable Markov decision problems,
in \emph{Advances in Neural Information Processing Systems} \textbf{19},
2006, pp.~1369-1376.

\bibitem{Todorov09}
E. Todorov,
Efficient computation of optimal actions,
\emph{Proc. Natl. Acad. Sci. USA} \textbf{106} (2009), 11478-11483.

\bibitem{TrelatBook}
E. Tr\'elat,
\emph{Control in Finite and Infinite Dimension},
SpringerBriefs in PDEs and Data Science, Springer, 2024.

\bibitem{TzenRaginsky19}
B. Tzen and M. Raginsky,
Theoretical guarantees for sampling and inference in generative models with latent diffusions,
in \emph{Proceedings of Machine Learning Research} \textbf{99},
2019, pp.~3084-3114.

\bibitem{Wong01}
R. Wong,
\emph{Asymptotic Approximations of Integrals},
SIAM, 2001.

\bibitem{YongZhou99}
J. Yong and X. Y. Zhou,
\emph{Stochastic Controls: Hamiltonian Systems and HJB Equations},
Springer, 1999.

\bibitem{ZhangChen22}
Q. Zhang and Y. Chen,
Path integral sampler: a stochastic control approach for sampling,
in \emph{International Conference on Learning Representations},
2022.


\end{thebibliography}
\end{document}